\documentclass[11pt]{amsart}

\usepackage{amsmath,amsthm,amssymb}
\usepackage[margin=1in]{geometry}
\usepackage{hyperref}
\usepackage{enumitem}
\usepackage{tikz}
\usepackage{booktabs}

\theoremstyle{plain}
\newtheorem{theorem}{Theorem}[section]
\newtheorem{lemma}[theorem]{Lemma}
\newtheorem{proposition}[theorem]{Proposition}
\newtheorem{corollary}[theorem]{Corollary}

\theoremstyle{definition}
\newtheorem{definition}[theorem]{Definition}
\newtheorem{remark}[theorem]{Remark}
\newtheorem{fact}[theorem]{Fact}
\newtheorem{conjecture}[theorem]{Conjecture}

\newcommand{\Z}{\mathbb{Z}}
\newcommand{\N}{\mathbb{N}}

\begin{document}

\title[Semiprime and sphenic Egyptian fractions]
{Every natural number is a sum of distinct semiprime unit fractions}

\author{Shisheng Li}
\email{shisheng@mail.ustc.edu.cn}
\date{June 2026}

\begin{abstract}
We prove that every natural number is a finite sum of distinct unit fractions
whose denominators are \emph{semiprimes} (products of two distinct primes).  This
is the $\omega=2$ integer case of a problem of Erd\H{o}s and Graham, stated only as a
conjecture by Butler, Erd\H{o}s and Graham (\emph{Integers} \textbf{15} (2015),
A51), who proved the $\omega=3$ analogue.  Counterintuitively the problem hardens
as $\omega$ decreases --- the induction's feed thins --- so $\omega=2$ is the hard
case; our proof adapts the Butler--Erd\H{o}s--Graham induction to this thin-feed
regime, where the entire content of the induction step reduces to an explicit
onset inequality $Y_0(N)\le\min\{\beta(N),\beta'(N)\}$, proved for all $N\ge10$ by
Olson's addition theorem and elementary Chebyshev bounds above a finite,
machine-checked base range.  The same engine extends to the rationals:
for every squarefree $b$, every $a/b$ above an explicit threshold
$\min\{B_{N_b}/6,\,\tfrac15\}$ is $\omega=2$ representable, unconditionally.  As an
application we give the first complete proof of the rational $\omega=3$ statement
--- every $a/b$ with squarefree $b$ is a sum of distinct \emph{sphenic} unit
fractions --- that Butler, Erd\H{o}s and Graham conjectured but left unpublished; a
descent settles every $\omega\ge3$.  What remains open is the $\omega=2$ regime below this
threshold, which we reduce to a single explicit conjecture --- that the gap-free
floor of a semiprime subset-sum set tends to zero.
This work is a human--AI collaboration: AI tools (notably Anthropic's Claude,
used through Claude Code) contributed substantially to the Lean formalisation, the
experiments, and the writing; correspondingly, every result is machine-checked in
Lean~4 / Mathlib (no \texttt{sorry}; two cited classical axioms, plus the
\texttt{native\_decide} compiler-trust base for the finite computations), so its
correctness is independent of the tools used.
\end{abstract}

\maketitle

\section{Introduction}

\subsection{Egyptian fractions with restricted denominators}
An \emph{Egyptian fraction} representation of a positive rational $a/b$ is an
expression
\[
\frac{a}{b}=\frac{1}{n_1}+\frac{1}{n_2}+\cdots+\frac{1}{n_k},
\qquad n_1<n_2<\cdots<n_k,
\]
as a sum of \emph{distinct} unit fractions.  Such representations always exist
(the greedy algorithm of Fibonacci produces one), and a rich vein of problems
arises when one constrains the denominators $n_i$.  Erd\H{o}s and Graham
\cite{ErdosGraham1980} asked, for a squarefree denominator $b$, whether every
$a/b$ admits an Egyptian-fraction representation in which each $n_i$ is a product
of exactly $\omega$ distinct primes.  This is recorded as
problem~\#306 in Bloom's \texttt{erdosproblems.com}~\cite{Bloom306} and as~D11 in
Guy's \emph{Unsolved Problems in Number Theory}~\cite{Guy}.

For $\omega=3$, Butler, Erd\H{o}s and Graham \cite{BEG2015} (henceforth
\emph{BEG15}) proved the integer case in full:

\begin{theorem}[BEG15, Theorem 1]
Any natural number can be written as an Egyptian fraction in which each
denominator is the product of three distinct primes.
\end{theorem}

In the same paper they remark that ``similar arguments could be used'' to handle
$\omega\ge 4$, but for $\omega=2$ they write only: ``\emph{We also conjecture
that a similar result holds for $\omega=2$}'' (\cite{BEG2015}, p.~2).  The
difficulty is structural: as $\omega$ decreases, the combinatorial ``feed''
driving the BEG15 induction becomes thinner, and at $\omega=2$ it degenerates to
the very thinnest object the method can use --- the set of subset sums of
single-prime reciprocals.  It is precisely this degeneration that BEG15 declined
to push through; we isolate the exact mechanism, and our way past it, in
\S\ref{ssec:new}.

\subsection{Main result}
We resolve the integer case of the $\omega=2$ conjecture.

\begin{theorem}\label{thm:main}
Every natural number $m\ge 1$ can be written as a finite sum of distinct unit
fractions
\[
m=\sum_{i}\frac{1}{n_i},\qquad n_1<n_2<\cdots<n_k,
\]
in which each denominator $n_i$ is a product of two distinct primes (a
semiprime).
\end{theorem}

Equivalently, the reciprocals $\{1/(pq):p\ne q\text{ prime}\}$ form a universal
additive basis for $\N$ in \emph{distinct} unit fractions.  This is the integer
($b=1$) case of the Erd\H{o}s--Graham problem above, and exactly the $\omega=2$
statement conjectured in BEG15.

Theorem~\ref{thm:main} is the engine; the rest of the paper drives it as far as it
goes. \S\ref{sec:rational} extends to all squarefree denominators above an explicit
threshold (Theorems~\ref{thm:above},~\ref{thm:uniform}); \S\ref{sec:omega3} proves
the full rational $\omega=3$ statement (Theorem~\ref{thm:omega3}) that Butler,
Erd\H{o}s and Graham conjectured but left unpublished, by lifting the
$\omega=2$ threshold theorem, and a descent extends it to all $\omega\ge3$
(Corollary~\ref{cor:omegak}); and \S\ref{sec:open} isolates the single remaining
open statement as the gap-free-floor conjecture for an explicit subset-sum set.
All results are unconditional.  A reader wanting only the
integer theorem may read \S\S\ref{sec:prelim}--\ref{sec:comp}; \S\S\ref{sec:rational}--\ref{sec:omega3}
are the rational extension, and \S\ref{sec:open} the open core.

A concrete instance for $m=1$ was already known: Johnson \cite{Johnson1978}
(reproduced in \cite{Guy} and in BEG15) exhibited $1$ as a sum of $48$ unit
fractions with semiprime denominators.  What was missing --- and what we supply
--- is the assertion for \emph{all} natural numbers, by an effective but
non-constructive argument paralleling the $\omega=3$ case.

\subsection{Where BEG15 stops, and the new ingredient}\label{ssec:new}
Since this paper pushes a known method through its hardest case, we state at the
outset exactly where the published $\omega=3$ argument halts and how we get past
it.

\paragraph{The skeleton is BEG15's.}
We use the BEG15 machine unchanged: the splitting recursion
\[
L^2(N+1)=L^1(N)+p_{N+1}\,L^2(N)
\tag{S}
\]
(Section~\ref{sec:step}), the propagation of a central contiguous block of subset
sums by induction on $N$ (Lemma~\ref{lem:block}), Olson's addition theorem
\cite{Olson1968} to control residues modulo the incoming prime $q=p_{N+1}$, and
the Mertens-type growth $B_N\to\infty$ that makes the induction's covering
intervals eventually swallow every integer.  None of these ingredients is new.

\paragraph{The precise obstruction.}
What varies with $\omega$ is the \emph{feed}: the leading term $L^{\omega-1}(N)$
of the recursion, which must realize every residue class modulo $q$.  For
$\omega\ge 3$ the feed $L^{\omega-1}(N)$ is already a rich subset-sum set, and
BEG15 fill the target block by writing this term as shifted copies of an inner,
residue-complete subset-sum set $P(S_{\dots})$ and merging the copies into one
contiguous run by a bounded max-gap (``runs connect'') estimate; this requires
the inner set to be \emph{both} residue-complete \emph{and} low-gap.  At
$\omega=2$ the feed collapses to the thinnest object the method admits,
$L^1(N)=P(D_1(N))$, the subset sums of the $N$ single-prime atoms $P_N/p_i$, and
the inner set on which the BEG15 decomposition rests degenerates to a
\emph{single atom} --- neither residue-complete nor low-gap.  The runs-connect
argument then has nothing to connect, and the $\omega=3$ proof does not
transcribe.  This is precisely why BEG15 could write that ``similar arguments''
handle $\omega\ge 4$ (where the feed only gets richer) yet recorded $\omega=2$
only as a \emph{conjecture}.

\paragraph{The new ingredient.}
We do not try to restore internal structure to the thin feed.  Instead we observe
that in the $\omega=2$ recursion the covering window is always at least as long as
the entire feed range: its length is $r_{hi}-r_{lo}\ge\sigma^1(N)=\max L^1(N)$
(Lemma~\ref{lem:seam}).  An interval that long cannot lie strictly inside
$(0,\sigma^1(N))$, so it must contain one of the feed endpoints $0$ or
$\sigma^1(N)$.  Hence the feed $L^1(N)$ is used \emph{directly} as the
residue-complete object --- its residue-completeness modulo $q$ is just Olson
applied to the single-prime atoms (Fact~\ref{fact:olson-feed}) --- and the only
thing that can fail is whether the feed attains every residue \emph{soon enough},
i.e.\ within the available strip.  All remaining difficulty is thereby
concentrated into the single endpoint inequality
\[
Y_0(N)\ \le\ \beta(N),
\]
with $Y_0(N)$ the residue-completeness onset radius of the feed and $\beta(N)$ the
strip length (Section~\ref{sec:analytic}).  Proving this for every $N\ge 10$ ---
exactly on a finite range, and by an elementary Chebyshev/monotonicity estimate
beyond it --- is the content the $\omega=2$ case demands and that BEG15 did not
supply.

\subsection{Scope and computations}\label{ssec:scope}
\begin{itemize}[leftmargin=1.5em]
\item Theorem~\ref{thm:main} settles the $b=1$ integer case of the problem, and
Section~\ref{sec:rational} extends the \emph{same} engine to every squarefree $b$
\emph{above} an explicit threshold (Theorems~\ref{thm:above},~\ref{thm:uniform}):
$a/b$ is unconditionally representable once $a/b\ge\min\{B_{N_b}/6,\,\tfrac15\}$.  What
remains \textbf{open} is only the \emph{below-threshold} bottom edge of the
$\omega=2$ problem --- rationals with small numerator and large squarefree
denominator, where the target falls below the central block into the low range of
$L^2(N)$.  This is the same end-of-interval difficulty BEG15 leave unresolved even
for $\omega=3$ (\cite{BEG2015}, p.~8: ``the structure of $L_{n+3}(n)$ is not well
understood at the ends of the interval'').  We resolve it for $\omega=3$
(Theorem~\ref{thm:omega3}, by a lift that bypasses the deep core) and isolate the
remaining $\omega=2$ deep core as one explicit gap-free-floor conjecture
$(\mathrm{GFF})$ (\S\ref{sec:open}).

\item The proof is partly \emph{computational}, but only in a bounded and
independently reproducible way.  Three of its ingredients are established by
exact finite computation: the base case at $N_0=10$; the inequality
$Y_0(N)\le\min\{\beta(N),\beta'(N)\}$ for $10\le N\le 299$; and the residue-completeness
of the feed (Fact~\ref{fact:olson-feed}) in the four cases $N=10,11,12,13$ where
Olson's bound does not apply.  Every such
computation is exact-integer arithmetic; none uses floating point in a
load-bearing way.  The remaining ingredients are analytic but \emph{elementary}:
Olson's addition theorem, the elementary Chebyshev prime bounds, and the
monotonicity of $A_N$.  The rational extension of Section~\ref{sec:rational} adds
only a finite $47$-pair sliver check; its uniform threshold reuses the $N_0=10$
base case above (no separate large computation).  We
describe all computations and their independent re-verification in
Section~\ref{sec:comp}.
\end{itemize}

\subsection{Relation to other recent work}
Two recent preprints address \emph{different} restrictions on the
denominator and do not bear on the present problem: Czenky et al.\
\cite{Czenky2025} study representations using a \emph{fixed finite} set of primes
(motivated by fusion categories), and a recent preprint \cite{TwoPrimePowers}
studies denominators of the shape $p^aq^b$ for two \emph{fixed} primes.  Neither
proves the $\omega=2$ universal statement for varying prime denominators.  As of
this writing the $\omega=2$ integer case appears unpublished, and the problem is
listed as open \cite{Bloom306}.

\section{Notation and the reduction \texorpdfstring{$(\star)$}{(*)}}
\label{sec:prelim}

Throughout, $p_n$ denotes the $n$-th prime ($p_1=2$, $p_2=3,\dots$), and
\[
P_N=p_1p_2\cdots p_N
\]
is the $N$-th primorial.  We work with subset sums of two ``atom'' sets.

\begin{definition}\label{def:atoms}
For $1\le r\le N$ put
\[
D_r(N)=\Bigl\{\,\tfrac{P_N}{p_{i_1}\cdots p_{i_r}} : 1\le i_1<\cdots<i_r\le N\,\Bigr\},
\]
the set of $\binom{N}{r}$ integers obtained from $P_N$ by deleting $r$ distinct
prime factors.  Given a finite set $X=\{x_1,\dots,x_m\}$ of integers, write
\[
P(X)=\Bigl\{\textstyle\sum_{i=1}^m \varepsilon_i x_i:\varepsilon_i\in\{0,1\}\Bigr\}
\]
for its set of subset sums.  Put
\[
L^r(N)=P\bigl(D_r(N)\bigr),\qquad
\sigma^r(N)=\sum_{x\in D_r(N)} x,
\]
so that $\sigma^r(N)=\max L^r(N)$ and (since $0\in L^r(N)$ and $L^r(N)$ is
symmetric about $\sigma^r(N)/2$) $L^r(N)\subseteq[0,\sigma^r(N)]$.
\end{definition}

We use the two cases $r=1$ and $r=2$:
\[
D_1(N)=\{P_N/p_i : 1\le i\le N\},\qquad
D_2(N)=\{P_N/(p_ip_j):1\le i<j\le N\}.
\]
We call $D_1(N)$ the \emph{feed} and $L^1(N)$ the \emph{feed sums}; $D_2(N)$ is
the \emph{semiprime atom set} for the first $N$ primes.

\begin{remark}
Dividing an element $P_N/(p_ip_j)\in D_2(N)$ by $P_N$ gives the unit fraction
$1/(p_ip_j)$ with semiprime denominator.  Thus subset sums of $D_2(N)$ correspond
exactly to sums of distinct semiprime unit fractions whose primes lie among the
first $N$. This yields the basic reduction.
\end{remark}

\begin{proposition}[reduction $(\star)$]\label{prop:star}
A positive rational $m$ is a sum of distinct unit fractions with semiprime
denominators using primes among the first $N$ if and only if
\begin{equation}
P_N\, m\in L^2(N).
\tag{$\star$}\label{eq:reduction}
\end{equation}
In particular, if a natural number $m$ satisfies $P_N m\in L^2(N)$ for some $N$,
then $m$ has a representation as in Theorem~\ref{thm:main}.
\end{proposition}

\begin{proof}
A subset $S\subseteq D_2(N)$ with $\sum S=P_N m$ corresponds, after dividing by
$P_N$, to a set of distinct semiprime reciprocals $1/(p_ip_j)$ summing to $m$;
the denominators are distinct semiprimes since the atoms $P_N/(p_ip_j)$ are
distinct.  Conversely a representation $m=\sum 1/(p_ip_j)$ with primes among the
first $N$ scales to a subset of $D_2(N)$ summing to $P_N m$.
\end{proof}

It is convenient to normalize $\sigma^2(N)$ by $P_N$.  Set
\[
A_N:=\sum_{i\le N}\frac{1}{p_i},\qquad
B_N:=\frac{\sigma^2(N)}{P_N}=\sum_{1\le i<j\le N}\frac{1}{p_ip_j}
=\frac12\!\left(A_N^2-\sum_{i\le N}\frac{1}{p_i^2}\right).
\]
By Mertens' theorem $A_N\to\infty$, while $\sum_i p_i^{-2}$ converges; hence
$B_N\to\infty$, and indeed $B_N\sim\frac12(\log\log N)^2$.  Below we use only that
$A_N$ is increasing and unbounded.

Table~\ref{tab:notation} collects the recurring notation.

\begin{table}[h]
\centering
\small
\begin{tabular}{@{}lll@{}}
\toprule
symbol & meaning & defined \\
\midrule
$p_i$ & the $i$-th prime ($p_1=2$); $q=p_{N+1}$ & \S\ref{sec:prelim}\\
$P_N$ & primorial $p_1\cdots p_N$ & \S\ref{sec:prelim}\\
$D_r(N)$ & atoms $P_N/(p_{i_1}\cdots p_{i_r})$ ($r$ primes deleted) & Def.~\ref{def:atoms}\\
$L^r(N)$ & subset sums of $D_r(N)$; $\sigma^r(N)=\max L^r(N)$ & Def.~\ref{def:atoms}\\
$A_N$ & $\sum_{i\le N}1/p_i$ & \eqref{eq:reduction} ff.\\
$B_N$ & $\sigma^2(N)/P_N=\sum_{i<j}1/(p_ip_j)$ (normalized) & \S\ref{sec:prelim}\\
$\beta(N),\beta'(N)$ & integer strip lengths $a_{N+1}-q\,a_N$, etc. & Def.~\ref{def:Y0}\\
$Y_0(N)$ & residue-completeness onset radius of the feed & Def.~\ref{def:Y0}\\
$\gamma_N,\gamma_\infty$ & gap-free floor of $L^2(N)$; its uniform bound & \S\ref{sec:rational}\\
$\gamma_{\mathrm{ex}}(M,r)$ & gap-free floor with prime $r$ deleted & \S\ref{sec:omega3}\\
$w_k$ & floor-recurrence cost $=Y_0(k-1)/P_{k-1}$ & \eqref{eq:wmin}\\
$B_k(N)$ & $\sum_{i_1<\cdots<i_k}1/(p_{i_1}\cdots p_{i_k})$; $B_2=B_N$ & \S\ref{sec:omega3}\\
$B_\Pi(M,r)$ & $B_2$ over the primes $\{p_1,\dots,p_M\}\setminus\{r\}$ & \S\ref{sec:omega3}\\
$N_b$ & index of the largest prime factor of $b$ & \S\ref{sec:rational}\\
\bottomrule
\end{tabular}
\caption{Recurring notation.}
\label{tab:notation}
\end{table}

\section{The induction skeleton}\label{sec:skeleton}

The engine of the proof is the following block-containment statement, the
$\omega=2$ analogue of BEG15's Lemma~1.

\begin{lemma}[central block]\label{lem:block}
For every $N\ge 10$,
\[
L^2(N)\ \supseteq\ \Bigl[\bigl\lceil\tfrac16\sigma^2(N)\bigr\rceil,\
\bigl\lfloor\tfrac56\sigma^2(N)\bigr\rfloor\Bigr]\cap\Z .
\]
That is, $L^2(N)$ contains \emph{every} integer in the central block
$\bigl[\tfrac16\sigma^2(N),\tfrac56\sigma^2(N)\bigr]$.
\end{lemma}

Lemma~\ref{lem:block} is proved by induction on $N$, with:
\begin{itemize}[leftmargin=1.5em]
\item \textbf{Base case} $N_0=10$ (Proposition~\ref{prop:base}, a direct exact
computation), and
\item \textbf{Induction step} $N\to N+1$, valid for all $N\ge 10$
(Section~\ref{sec:step}).
\end{itemize}
We first record the base case and then deduce Theorem~\ref{thm:main}, deferring
the step to the next section.

\begin{proposition}[base case]\label{prop:base}
With $N_0=10$,
\[
P_{10}=6{,}469{,}693{,}230,\qquad \sigma^2(10)=6{,}166{,}988{,}769,
\]
and
\[
L^2(10)\ \supseteq\ \bigl[\,1{,}027{,}831{,}462,\ 5{,}139{,}157{,}307\,\bigr]\cap\Z
=\Bigl[\bigl\lceil\tfrac16\sigma^2(10)\bigr\rceil,
\bigl\lfloor\tfrac56\sigma^2(10)\bigr\rfloor\Bigr]\cap\Z,
\]
with no integer of the block missing.
\end{proposition}

\begin{proof}
This is verified by exact subset-sum reachability over the $\binom{10}{2}=45$
atoms of $D_2(10)$; see Section~\ref{sec:comp} for the method (a packed bitset of
$\sigma^2(10)+1$ bits, $\approx 771$\,MB) and its independent re-verification.
Every integer in the stated block is a subset sum of $D_2(10)$; $0$ integers are
missing.
\end{proof}

\begin{remark}[why the base is $N_0=10$, not $9$]\label{rem:base9}
One might hope to seed at $N_0=9$, where $L^2(9)$ also covers its central block
$[33{,}520{,}589,\,167{,}602{,}941]$ exactly.  However the induction step
$N\to N+1$ requires the inequality $Y_0(N)\le \beta(N)$ of
Section~\ref{sec:analytic}, and this \emph{fails} at $N=9$:
$Y_0(9)=91{,}525{,}280>\beta(9)=55{,}734{,}381$.  The inequality $Y_0(N)\le \beta(N)$
holds for all $N\ge 10$, so seeding directly at $N_0=10$ --- where the step is
valid from the start --- repairs this and the induction runs cleanly upward.
\end{remark}

We isolate the elementary overlap estimate that drives the interval-covering
argument.

\begin{proposition}[explicit overlap]\label{prop:overlap}
For every $N\ge 1$ the sequence $B_N$ is strictly increasing, and for every
$N\ge 10$,
\[
B_{N+1}\le 5B_N,\qquad\text{equivalently}\qquad B_{N+1}-B_N\le 4B_N .
\]
\end{proposition}

\begin{proof}
The increment is
\[
B_{N+1}-B_N=\sum_{i\le N}\frac1{p_i\,p_{N+1}}=\frac{A_N}{p_{N+1}}>0,
\]
since the only pairs counted by $B_{N+1}$ and not by $B_N$ are $\{i,N+1\}$ with
$i\le N$.  Hence $B_N$ is strictly increasing, and by direct computation
$B_{10}=0.9532\ldots>\tfrac14$, so $B_N\ge B_{10}>\tfrac14$ for all $N\ge 10$;
thus $4B_N>1$.  On the other hand $p_i\ge 2$ gives $A_N=\sum_{i\le N}1/p_i\le N/2$,
while the $(N+1)$-st prime satisfies $p_{N+1}\ge N+2$ (among $1,\dots,p_{N+1}$
there are exactly $N+1$ primes and $1$ is not prime, so $p_{N+1}\ge N+2$).
Therefore
\[
B_{N+1}-B_N=\frac{A_N}{p_{N+1}}\le\frac{N/2}{N+2}<\frac12<1<4B_N ,
\]
which is the claim.
\end{proof}

\begin{proof}[Proof of Theorem~\ref{thm:main}, assuming Lemma~\ref{lem:block}]
Fix a natural number $m\ge 1$.  By $(\star)$ (Proposition~\ref{prop:star}), $m$
is representable as soon as $P_N m\in L^2(N)$ for some $N$.  By
Lemma~\ref{lem:block}, $L^2(N)$ contains every integer in
$[\tfrac16\sigma^2(N),\tfrac56\sigma^2(N)]$; since $P_N m$ is an integer, it
suffices that
\[
\tfrac16\,\sigma^2(N)\ \le\ P_N m\ \le\ \tfrac56\,\sigma^2(N),
\quad\text{equivalently}\quad
\tfrac16 B_N\ \le\ m\ \le\ \tfrac56 B_N,
\]
for some $N\ge 10$ (dividing by $P_N$ and using $B_N=\sigma^2(N)/P_N$).

Consider the intervals $I_N:=[\tfrac16 B_N,\tfrac56 B_N]$ for $N\ge 10$.  The
right endpoint of $I_N$ is at or above the left endpoint of $I_{N+1}$ exactly
when $\tfrac56 B_N\ge\tfrac16 B_{N+1}$, i.e.\ $B_{N+1}\le 5B_N$, which holds for
all $N\ge 10$ by Proposition~\ref{prop:overlap}.  Consecutive intervals
therefore overlap, and since $B_N$ is increasing and unbounded
($B_N\to\infty$), the union $\bigcup_{N\ge 10} I_N$ is the half-line
$[\tfrac16 B_{10},\infty)$.  As
$B_{10}=\sigma^2(10)/P_{10}=6{,}166{,}988{,}769/6{,}469{,}693{,}230=0.9532\ldots$
gives $\tfrac16 B_{10}=0.1589<1$, this half-line
contains every integer $m\ge 1$.  (For instance $m=1$ first lies in $I_N$ at
$N=17$, where $B_{17}=1.2162\ldots$ gives $I_{17}=[0.2027,1.0135]\ni 1$.)  Every
integer $m\ge 1$ thus lies in some $I_N$, hence is representable.
\end{proof}

\begin{remark}
As in BEG15, the bound on $N$ needed to represent a given $m$ is enormous and the
argument is non-constructive: $m$ enters $I_N$ only once
$B_N\sim\frac12(\log\log N)^2\gtrsim\frac65 m$, i.e.\ at $N$ of order
$\exp\exp(O(\sqrt m))$.  Theorem~\ref{thm:main} asserts existence of a
representation, not a small one.
\end{remark}

\section{The induction step \texorpdfstring{$N\to N+1$}{N to N+1}}
\label{sec:step}

We now prove the step in Lemma~\ref{lem:block}: assuming
\[
L^2(N)\supseteq\bigl[a_N,b_N\bigr]\cap\Z,\qquad
a_N=\bigl\lceil\tfrac16\sigma^2(N)\bigr\rceil,\quad
b_N=\bigl\lfloor\tfrac56\sigma^2(N)\bigr\rfloor,
\]
we deduce the same with $N+1$ in place of $N$, for every $N\ge 10$.

\subsection{The splitting recursion}
Partitioning $D_2(N+1)$ according to whether the largest available prime
$p_{N+1}$ is deleted gives the disjoint union
\[
D_2(N+1)=D_1(N)\ \sqcup\ p_{N+1}\cdot D_2(N).
\tag{R}
\]
Indeed, an element of $D_2(N+1)$ is $P_{N+1}$ divided by two distinct primes
among $p_1,\dots,p_{N+1}$.  If neither deleted prime is $p_{N+1}$ we get
$P_{N+1}/(p_ip_j)=p_{N+1}\cdot P_N/(p_ip_j)\in p_{N+1}D_2(N)$; if one deleted
prime is $p_{N+1}$ we get $P_{N+1}/(p_{N+1}p_i)=P_N/p_i\in D_1(N)$.  Taking subset
sums of a disjoint union, and recalling $P(A\sqcup B)=P(A)+P(B)$ (sumset),
\begin{equation}
L^2(N+1)=L^1(N)+p_{N+1}\,L^2(N).
\tag{S}\label{eq:rec}\end{equation}
This is the $r=2$ instance of BEG15's recursion (1).  Likewise summing all of
$D_2(N+1)$ gives $\sigma^2(N+1)=\sigma^1(N)+p_{N+1}\sigma^2(N)$.

\subsection{Reformulating the step as a covering problem}
The objects in play, all at level $N$: the incoming prime $q=p_{N+1}$; the
inductive block $[a_N,b_N]\subseteq L^2(N)$ and its $q$-scaling
$[r_{lo},r_{hi}]=q[a_N,b_N]$; the feed range $[0,\sigma^1(N)]$; and, for each target
$x$, the feed window $W(x)$ defined below.
By the induction hypothesis $L^2(N)\supseteq[a_N,b_N]$, so
$p_{N+1}L^2(N)=q\,L^2(N)$ contains \emph{every multiple of $q$} in
$[r_{lo},r_{hi}]$, where
\[
r_{lo}=q\,a_N=q\bigl\lceil\tfrac16\sigma^2(N)\bigr\rceil,\qquad
r_{hi}=q\,b_N=q\bigl\lfloor\tfrac56\sigma^2(N)\bigr\rfloor.
\]
By (S), an integer $x$ is in $L^2(N+1)$ provided there is a feed sum
$u\in L^1(N)$ with $u\equiv x\pmod q$ and $r_{lo}\le x-u\le r_{hi}$ (then
$x=u+(x-u)$ with $x-u$ a multiple of $q$ in $[r_{lo},r_{hi}]$, hence
$x-u\in qL^2(N)$).  Equivalently, $x\in L^2(N+1)$ if the feed window
\[
W(x):=[\,x-r_{hi},\,x-r_{lo}\,]\cap[0,\sigma^1(N)]
\]
contains an element of $L^1(N)$ congruent to $x\bmod q$.  We must show this holds
for every integer $x$ in the target block $[a_{N+1},b_{N+1}]$.

\subsection{The window always spans an endpoint of the feed}
\emph{This is the observation the $\omega=2$ case turns on.}
The window $[x-r_{hi},x-r_{lo}]$ has length $r_{hi}-r_{lo}$.  We show this length
is at least $\sigma^1(N)$, so that the window cannot lie strictly inside the open
feed range $(0,\sigma^1(N))$: an interval of length $\ge\sigma^1$ cannot be a
subset of an open interval of length $\sigma^1$.  Hence
$W(x)=[x-r_{hi},x-r_{lo}]\cap[0,\sigma^1(N)]$ always contains one of the two
endpoints $0$ or $\sigma^1(N)$.  The inequality $r_{hi}-r_{lo}\ge\sigma^1(N)$ is
the second seam inequality, proved with large margin in
Lemma~\ref{lem:seam}.  This splits the target block into three contiguous
pieces (Figure~\ref{fig:strips}).

\begin{figure}[h]
\centering
\begin{tikzpicture}[x=0.95cm,y=0.9cm,font=\footnotesize]
  \def\B{7.2}
  \node[anchor=west,font=\footnotesize\itshape] at (0,3.75){target block, partitioned by the seam};
  \draw[thick] (0,3.4)--(\B,3.4);
  \foreach \x in {0,2.4,4.8,\B} \draw (\x,3.5)--(\x,3.3);
  \node[below] at (0,3.3){$a_{N+1}$};
  \node[below] at (2.4,3.3){$\sigma^1{+}r_{lo}$};
  \node[below] at (4.8,3.3){$r_{hi}$};
  \node[below] at (\B,3.3){$b_{N+1}$};
  % three feed snapshots: window inside [0, sigma^1]
  \def\S{4.8}
  \foreach \y/\wa/\wb/\desc in {%
      2.05/0/2.6/{bottom strip: $W(x)=[0,\,x-r_{lo}]$, pinned at $0$},
      1.15/0/4.8/{full-feed block: $W(x)=[0,\,\sigma^1]$},
      0.25/2.2/4.8/{top strip: $W(x)=[x-r_{hi},\,\sigma^1]$, pinned at $\sigma^1$}}{
     \draw (0,\y)--(\S,\y);
     \draw (0,\y+0.09)--(0,\y-0.09);
     \draw (\S,\y+0.09)--(\S,\y-0.09);
     \fill[black!15] (\wa,\y-0.075) rectangle (\wb,\y+0.075);
     \node[right] at (\S+0.25,\y){\desc};
  }
  \node[below] at (0,0.16){\scriptsize $0$};
  \node[below] at (\S,0.16){\scriptsize $\sigma^1$};
\end{tikzpicture}
\caption{The induction step covers the target block in three pieces.  The covering
window $W(x)$ has length $\ge\sigma^1$, so it cannot fit strictly inside the feed
range $(0,\sigma^1)$ and must meet an endpoint: pinned at $0$ in the bottom strip,
at $\sigma^1$ in the top strip, and equal to the whole feed in between.  Each piece
is covered once the feed $L^1(N)$ hits every residue mod $q$ soon enough
(Proposition~\ref{prop:step-reduce}).}
\label{fig:strips}
\end{figure}
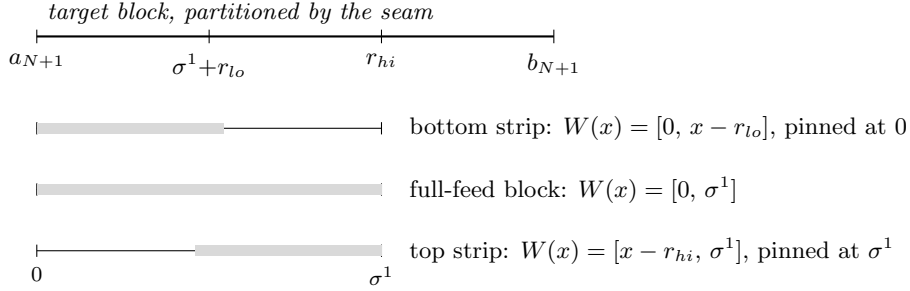

\paragraph{Full feed.}
For $x\in[\sigma^1(N)+r_{lo},\,r_{hi}]$ the window is the whole feed range,
$W(x)=[0,\sigma^1(N)]$.  Here $x$ is covered as soon as the feed sums $L^1(N)$ are
\emph{residue-complete} modulo $q$, i.e.\ $L^1(N)\bmod q=\Z_q$; this is
Fact~\ref{fact:olson-feed} below.

\paragraph{Bottom strip.}
For $x\in[a_{N+1},\,\sigma^1(N)+r_{lo})$ the window touches $0$:
$W(x)=[0,\,x-r_{lo}]$, and $x-r_{lo}\ge a_{N+1}-r_{lo}=:\beta(N)$.  Thus
$W(x)\supseteq[0,\beta(N)]$.  So $x$ is covered for \emph{all} such $x$ provided the
feed sums lying in the initial segment $[0,\beta(N)]$ already hit every residue mod
$q$, i.e.\ provided $\beta(N)\ge Y_0(N)$, where $Y_0(N)$ (Definition~\ref{def:Y0}) is
the smallest $y$ for which $L^1(N)\cap[0,y]$ is residue-complete.

\paragraph{Top strip.}
For $x\in(r_{hi},\,b_{N+1}]$ the window touches $\sigma^1(N)$:
$W(x)=[x-r_{hi},\,\sigma^1(N)]$.  The feed $L^1(N)$ is symmetric under
$u\mapsto\sigma^1(N)-u$ (subset complementation in $D_1(N)$), so the top
``onset radius'' measured from $\sigma^1(N)$ downward equals $Y_0(N)$.  The
available top length is shortest at the right end $x=b_{N+1}$, where it equals
\[
\beta'(N):=\sigma^1(N)-\bigl(b_{N+1}-r_{hi}\bigr)=\sigma^1(N)-b_{N+1}+q\,b_N,
\]
so the top strip is covered provided $Y_0(N)\le \beta'(N)$.  We do \emph{not} claim
$\beta'(N)=\beta(N)$: the two strip lengths can differ by a bounded rounding term.  What
we use is that \emph{both} dominate the same explicit integer lower bound,
\[
\beta(N)\ \ge\ \tfrac16\sigma^1(N)-q,
\qquad
\beta'(N)\ \ge\ \tfrac16\sigma^1(N)-q
\tag{LB}
\]
(Lemmas~\ref{lem:Bbound} and~\ref{lem:Bprime}).  Hence the top strip is covered under
the \emph{same} sufficient condition $Y_0(N)\le\tfrac16\sigma^1(N)-q$ as the bottom strip,
and it is this common bound that Theorem~\ref{thm:Y0} establishes.

\begin{remark}[the step at $N=10$, $q=31$, normalized by $P_{10}$]
Concretely (Figure~\ref{fig:strips}): $\sigma^1(10)=1.5334\,P_{10}$, and the target
block is $[5.18,\,25.90]\,P_{10}$, split at $\sigma^1{+}r_{lo}=6.46\,P_{10}$ and
$r_{hi}=24.63\,P_{10}$ into the bottom strip, full-feed block, and top strip.  A
target $x=5.5\,P_{10}$ lies in the bottom strip; its window is
$W(x)=[0,\,0.575\,P_{10}]$, pinned at $0$, and is covered because the feed reaches
every residue mod $31$ already by $Y_0(10)=0.232\,P_{10}$, within the available
$\beta(10)=0.256\,P_{10}$.  The margin $0.024$ is the tightest in the entire
induction (Appendix~\ref{app:data}).
\end{remark}

\subsection{The seams cover, for every \texorpdfstring{$N\ge10$}{N>=10}}
For these three pieces to tile the target block $[a_{N+1},b_{N+1}]$ we need the
ordering
\[
r_{lo}\ <\ a_{N+1}\ \le\ \sigma^1(N)+r_{lo}\ \le\ r_{hi}\ \le\ b_{N+1}\ <\ r_{hi}+\sigma^1(N).
\tag{Seam}
\]
We prove all of (Seam) \emph{analytically}, for every $N\ge10$, with explicit
margins.

\begin{lemma}[seam ordering]\label{lem:seam}
For every $N\ge 10$ the chain \textup{(Seam)} holds; consequently the three
pieces cover $[a_{N+1},b_{N+1}]\cap\Z$ exactly.
\end{lemma}

\begin{proof}
Write $q=p_{N+1}$, $\sigma^1=\sigma^1(N)$, $\sigma^2=\sigma^2(N)$, and recall
$\sigma^2(N+1)=\sigma^1+q\,\sigma^2$, $a_M=\lceil\sigma^2(M)/6\rceil$,
$b_M=\lfloor 5\sigma^2(M)/6\rfloor$, $r_{lo}=q\,a_N$, $r_{hi}=q\,b_N$.  We use two
elementary facts valid for $N\ge 10$:
\begin{itemize}[leftmargin=1.5em]
\item[(P1)] $\sigma^1\ge\tfrac65$ (indeed $\sigma^1\ge\sigma^1(10)\approx 9.9\times10^9$);
\item[(P2)] $\sigma^1>6q$.  For $\sigma^1/q=P_NA_N/p_{N+1}$, Bertrand's postulate
$p_{N+1}<2p_N$ and $P_N=p_N P_{N-1}$ give $P_N/p_{N+1}>P_{N-1}/2$, so with
$A_N\ge A_{10}>1.53$ and $P_{N-1}\ge P_9=223{,}092{,}870$ we get
$\sigma^1/q>1.53\cdot P_9/2>10^{8}\gg 6$.
\end{itemize}

\emph{$r_{lo}<a_{N+1}$.}  Using $\lceil t\rceil\ge t$ and $\lceil t\rceil\le t+1$,
\[
a_{N+1}=\Bigl\lceil\tfrac{\sigma^1+q\sigma^2}{6}\Bigr\rceil\ge\tfrac{\sigma^1}6+\tfrac{q\sigma^2}6,
\qquad
r_{lo}=q\bigl\lceil\tfrac{\sigma^2}6\bigr\rceil\le\tfrac{q\sigma^2}6+q,
\]
so $a_{N+1}-r_{lo}\ge\tfrac16\sigma^1-q>0$ by (P2).

\emph{$a_{N+1}\le\sigma^1+r_{lo}$.}  Here
$a_{N+1}\le\tfrac{\sigma^1}6+\tfrac{q\sigma^2}6+1$ and
$\sigma^1+r_{lo}\ge\sigma^1+\tfrac{q\sigma^2}6$, so the difference is
$\ge\tfrac56\sigma^1-1>0$ by (P1).

\emph{$\sigma^1+r_{lo}\le r_{hi}$ (the key seam).}  Now
$r_{hi}-r_{lo}=q(b_N-a_N)\ge q\bigl(\tfrac56\sigma^2-1-\tfrac16\sigma^2-1\bigr)
=q\bigl(\tfrac23\sigma^2-2\bigr)$, so it suffices that
$\sigma^1\le q(\tfrac23\sigma^2-2)$; dividing by $P_N$ this is
\[
A_N\ \le\ \tfrac23 q\,B_N-\tfrac{2q}{P_N}.
\]
Since $2B_N=A_N^2-S_N$ with $S_N:=\sum_{i\le N}p_i^{-2}<\sum_p p^{-2}<0.4523$ and
$A_N\ge A_{10}>1.53$,
\[
\frac{A_N}{B_N}=\frac{2}{A_N-S_N/A_N}\le\frac{2}{1.53-0.4523/1.53}<1.63 .
\]
As $q\ge p_{11}=31$, we get
$\tfrac23 qB_N-A_N\ge(\tfrac23\cdot31-1.63)B_N>18B_N>18B_{10}>17$, which dwarfs
$2q/P_N<2\cdot31/P_{10}<10^{-7}$.  Hence $r_{hi}-r_{lo}-\sigma^1>17P_N>0$, which
in particular gives $r_{hi}-r_{lo}\ge\sigma^1$ (used in
Section~\ref{sec:step}).

\emph{$r_{hi}\le b_{N+1}$.}  Symmetrically
$r_{hi}=q\lfloor 5\sigma^2/6\rfloor\le\tfrac56 q\sigma^2$ and
$b_{N+1}\ge\tfrac56\sigma^1+\tfrac56 q\sigma^2-1$, so
$b_{N+1}-r_{hi}\ge\tfrac56\sigma^1-1>0$ by (P1).

\emph{$b_{N+1}<r_{hi}+\sigma^1$.}  By definition
$\beta'(N)=\sigma^1-(b_{N+1}-r_{hi})$, so this is $\beta'(N)>0$, which follows from
Lemma~\ref{lem:Bprime} and (P2): $\beta'(N)\ge\tfrac16\sigma^1-q>0$.

The displayed bounds partition $[a_{N+1},b_{N+1}]$ into the bottom strip
$[a_{N+1},\sigma^1+r_{lo})$, the full-feed block $[\sigma^1+r_{lo},r_{hi}]$, and
the top strip $(r_{hi},b_{N+1}]$.
\end{proof}

\begin{proposition}[reduction of the step]\label{prop:step-reduce}
Fix $N\ge 10$ and assume $L^2(N)\supseteq[a_N,b_N]\cap\Z$.  If
\begin{enumerate}[label=\textup{(\roman*)}]
\item the feed sums $L^1(N)$ are residue-complete modulo $q=p_{N+1}$, and
\item $Y_0(N)\le\min\{\beta(N),\beta'(N)\}$; by \textup{(LB)} it suffices that
$Y_0(N)\le\tfrac16\sigma^1(N)-q$,
\end{enumerate}
then $L^2(N+1)\supseteq[a_{N+1},b_{N+1}]\cap\Z$.
\end{proposition}

It remains to establish (i) and (ii) for all $N\ge 10$.  Ingredient (i) is
Fact~\ref{fact:olson-feed}; ingredient (ii) is the analytic heart, proved in
Section~\ref{sec:analytic}.

\subsection{Ingredient (i): residue-completeness of the feed}
\begin{fact}[feed residue-completeness]\label{fact:olson-feed}
For every $N\ge 10$ the feed sums $L^1(N)=P(D_1(N))$ are residue-complete modulo
$q=p_{N+1}$; that is, $L^1(N)\bmod q=\Z_q$.
\end{fact}

\begin{proof}
The atoms $D_1(N)=\{P_N/p_i:1\le i\le N\}$ reduce modulo $q=p_{N+1}$ to
$N$ \emph{distinct, nonzero} residues: nonzero because $q$ exceeds
$p_1,\dots,p_N$ and does not divide $P_N/p_i$; distinct because $P_N$ is a unit
mod $q$ and the inverses $p_i^{-1}\bmod q$ are distinct, so
$P_N/p_i\equiv P_N\,p_i^{-1}$ are distinct.  By Olson's theorem
(Theorem~\ref{thm:olson}) these $N$ residues have $P(D_1(N)\bmod q)=\Z_q$ whenever
$q=p_{N+1}<(N^2+3)/4$.  The actual primes satisfy this for all $N\ge 14$: directly
for $14\le N\le 27$, and for $N\ge 28$ by the elementary bound
$p_{N+1}<2(N+1)\log(N+1)<(N^2+3)/4$ (the second inequality first holds at $N=28$,
where $2\cdot 29\log 29=195.3<196.75$, and then for all larger $N$ as the gap is
increasing).  Olson's hypothesis \emph{fails} only for $N\le 13$ (at $N=13$,
$p_{14}=43=(13^2+3)/4$, not strictly less); among these the proof uses only
$N\ge 10$, and the four cases $N=10,11,12,13$ ($q\in\{31,37,41,43\}$) are settled
by exact subset-sum computation (Section~\ref{sec:comp}).
\end{proof}

\begin{theorem}[Olson \cite{Olson1968}; BEG15 Theorem 2]\label{thm:olson}
Let $X$ be a set of distinct nonzero residues modulo a prime $q$.  If
$q<(|X|^2+3)/4$, then $P(X)=\Z_q$.
\end{theorem}

Equivalently (the form we use below), a set of more than $\sqrt{4q-3}$ distinct
nonzero residues modulo $q$ has subset sums equal to all of $\Z_q$.

\section{The analytic heart: \texorpdfstring{$Y_0(N)\le\beta(N)$}{Y0(N) <= beta(N)}}
\label{sec:analytic}

\begin{definition}\label{def:Y0}
The \emph{residue-completeness onset radius} of the feed is
\[
Y_0(N):=\max_{r\in\Z_q}\ \min\bigl\{\textstyle\sum S:\ S\subseteq D_1(N),\
\textstyle\sum S\equiv r\!\!\pmod q\bigr\},\qquad q=p_{N+1},
\]
i.e.\ the smallest $y$ such that $L^1(N)\cap[0,y]$ meets every residue class
modulo $q$.  (By Fact~\ref{fact:olson-feed} the inner minima are finite, so
$Y_0(N)$ is well defined.)  Recall from \S\ref{sec:step} the two strip lengths
\[
\beta(N):=a_{N+1}-q\,a_N
=\bigl\lceil\tfrac16\sigma^2(N+1)\bigr\rceil
-q\bigl\lceil\tfrac16\sigma^2(N)\bigr\rceil ,
\qquad
\beta'(N):=\sigma^1(N)-b_{N+1}+q\,b_N .
\]
\end{definition}

\begin{lemma}\label{lem:Bbound}
$\beta(N)\ \ge\ \dfrac{\sigma^1(N)}{6}-q$.
\end{lemma}

\begin{proof}
Recall $\beta(N)=\lceil\tfrac16\sigma^2(N+1)\rceil-q\lceil\tfrac16\sigma^2(N)\rceil$
and $\sigma^2(N+1)=\sigma^1(N)+q\,\sigma^2(N)$.  Bound the two terms separately.
For the first, $\lceil t\rceil\ge t$ gives
\[
\Bigl\lceil\tfrac16\sigma^2(N+1)\Bigr\rceil\ \ge\ \tfrac16\sigma^2(N+1)
\ =\ \tfrac16\sigma^1(N)+\tfrac{q}{6}\sigma^2(N).
\]
For the second, $\lceil t\rceil\le t+1$ gives
\[
q\Bigl\lceil\tfrac16\sigma^2(N)\Bigr\rceil\ \le\ q\Bigl(\tfrac16\sigma^2(N)+1\Bigr)
\ =\ \tfrac{q}{6}\sigma^2(N)+q.
\]
Subtracting, the $\tfrac{q}{6}\sigma^2(N)$ terms cancel and
$\beta(N)\ge\tfrac16\sigma^1(N)-q$.
\end{proof}

\begin{lemma}\label{lem:Bprime}
$\beta'(N)\ \ge\ \dfrac{\sigma^1(N)}{6}-q$.
\end{lemma}

\begin{proof}
With $\sigma^2(N+1)=\sigma^1(N)+q\,\sigma^2(N)$ we have
$b_{N+1}\le\frac56\sigma^1(N)+\frac{5q}{6}\sigma^2(N)$, while
$q\,b_N\ge\frac{5q}{6}\sigma^2(N)-q$.  Hence
$\beta'(N)=\sigma^1(N)-b_{N+1}+q\,b_N\ge\frac16\sigma^1(N)-q$.
\end{proof}

\begin{theorem}\label{thm:Y0}
For all $N\ge 10$, $\ Y_0(N)\le\min\{\beta(N),\beta'(N)\}$.  Indeed for all $N\ge 300$
the stronger bound $Y_0(N)\le\tfrac16\sigma^1(N)-q$ holds, which by
Lemmas~\ref{lem:Bbound}--\ref{lem:Bprime} implies the former; the finite range
$10\le N\le 299$ is checked exactly.
\end{theorem}

We prove this in two regimes: an exact finite range and an elementary analytic
tail.

\subsection{Exact range \texorpdfstring{$10\le N\le 299$}{10 to 299}}
For each such $N$ the quantity $Y_0(N)$ is computed \emph{exactly} by a
Dijkstra-style shortest-path computation over the residue group $\Z_q$: maintain,
for each residue $r$, the minimal subset sum of $D_1(N)$ achieving $r$, updating
atom by atom.  The state space has size $q=p_{N+1}\le p_{300}=1{,}987$ in this
range, so the computation is polynomial (no exponential enumeration) and uses
exact integers.  One then checks $Y_0(N)\le\min\{\beta(N),\beta'(N)\}$ directly (both
strip lengths are computed exactly from the primorials).  All $290$ values pass;
see Section~\ref{sec:comp}.  (We note $Y_0(9)>\beta(9)$, consistent with
Remark~\ref{rem:base9}.)

\subsection{Elementary analytic tail \texorpdfstring{$N\ge 300$}{N >= 300}}
\label{ssec:tail}
Let $q=p_{N+1}$ and put $c_0:=\bigl\lfloor\sqrt{4q-3}\,\bigr\rfloor+1$, so that
$q<(c_0^2+3)/4$ strictly.\footnote{The Lean companion
(Appendix~\ref{app:lean}) closes the same tail from the sharper Rosser bound
\texttt{rs\_lower}, which lowers the crossover to $N\ge512$ and splits the finite head
into base, mid, and tail ranges; the theorem proved is identical --- only the threshold
and the finite/analytic partition differ.}  Let $C$ denote the $c_0$ \emph{smallest} atoms of the
feed, $C=\{P_N/p_i:\ i=N-c_0+1,\dots,N\}$.  Their residues mod $q$ are distinct
and nonzero (as in Fact~\ref{fact:olson-feed}), and $|C|=c_0$ with
$q<(c_0^2+3)/4$; for $N\ge 300$ one has $c_0\le N$, so $C\subseteq D_1(N)$.  By
Olson's theorem $P(C)=\Z_q$, so every residue is achieved by a subset of $C$,
whose total is at most $\sigma(C):=\sum_{i=N-c_0+1}^N P_N/p_i$.  Therefore
\[
Y_0(N)\ \le\ \sigma(C).
\tag{O}
\]
Dividing (O) by $P_N$, the desired bound $\sigma(C)\le\tfrac16\sigma^1(N)-q$ is
equivalent to
\[
\Sigma(N)\ :=\ \sum_{i=N-c_0+1}^{N}\frac{1}{p_i}\ +\ \frac{q}{P_N}
\ \le\ \frac{A_N}{6}.
\tag{Tail}
\]
We close (Tail) for all $N\ge 300$ using only the elementary Chebyshev prime bounds
\[
L(k):=\tfrac12 k\log k\ <\ p_k\ <\ 2k\log k=:U(k)\qquad(k\ge 6),
\tag{Ch}
\]
which are provable by the binomial-coefficient method and which we verify
directly for the (few) small indices used below, together with the monotonicity
of $A_N$ (explicit forms of all these prime estimates are classical,
\cite{RosserSchoenfeld1962}).

\begin{lemma}[elementary certified tail]\label{lem:tail-elem}
For all $N\ge 300$, inequality \textup{(Tail)} holds.  Consequently
$\sigma(C)\le\tfrac16\sigma^1(N)-q$ and $Y_0(N)\le\tfrac16\sigma^1(N)-q$.
\end{lemma}

\begin{proof}
Write $g(N):=2\sqrt{U(N+1)}+1$.

\emph{Lower bound for the right side.}  $A_N$ is strictly increasing
(Proposition~\ref{prop:overlap}), so $A_N\ge A_{300}$; and
$A_{300}=\sum_{i\le 300}1/p_i=2.2909\ldots$ is a finite sum, giving
$A_N/6\ge A_{300}/6=0.3818\ldots$

\emph{Upper bound for the left side.}  The $c_0$ summands are each $\le1/p_{N-c_0}$,
so the feed tail is $\le c_0/p_{N-c_0}$.  By (Ch), $q=p_{N+1}<U(N+1)$, hence
$c_0\le 2\sqrt q+1<2\sqrt{U(N+1)}+1=g(N)$; and since $L$ is increasing and
$c_0<g(N)$, $p_{N-c_0}>L(N-c_0)\ge L(N-g(N))$.  Thus the feed tail is
$<\tau(N):=g(N)/L(N-g(N))$.  The carry is $q/P_N\le U(N+1)/2^N=:\varepsilon(N)$,
with $\varepsilon(300)<3436/2^{300}<10^{-86}$.  Hence
$\Sigma(N)<\overline\Sigma(N):=\tau(N)+\varepsilon(N)$.

\emph{Monotonicity.}  $\varepsilon$ is decreasing.  For $\tau$, with $D:=N-g(N)$,
\[
  (\log\tau)'=\frac{g'(N)}{g(N)}-\bigl(1-g'(N)\bigr)\frac{L'(D)}{L(D)}
  \ \le\ \frac{g'(N)}{g(N)}-\frac{1-g'(N)}{D},
\]
using $L'(D)/L(D)=(\log D+1)/(D\log D)\ge 1/D$ and $1-g'(N)>0$.  Since $D+g(N)=N$,
the right side is $<0$ exactly when $g'(N)\,(D+g(N))<g(N)$, i.e.\ $g'(N)<g(N)/N$.
From $g(N)=2\sqrt{U(N+1)}+1$ with $U(k)=2k\log k$ one computes
$g'(N)=2(\log(N+1)+1)/\sqrt{U(N+1)}$, which is decreasing; at $N=300$,
$g'(300)=13.41/58.62<0.23$ while $g(300)/300>118/300>0.39$, and the ratio
$g'(N)/(g(N)/N)\to\tfrac12$.  Hence $g'(N)<g(N)/N$ throughout $[300,\infty)$, so
$(\log\tau)'<0$ and $\overline\Sigma$ is decreasing there.

\emph{Base evaluation at $N=300$.}  $U(301)=602\log 301<3436$, so
$g(300)<2\sqrt{3436}+1<118.3$, $D(300)>181.7$, and
$L(D(300))>\tfrac12\cdot181.7\cdot\log 181.7>472$; whence
$\tau(300)<118.3/472<0.2507$, and adding $\varepsilon(300)<10^{-86}$ gives
$\overline\Sigma(300)<0.2507$.

\emph{Conclusion.}  For all $N\ge300$,
\[
  \Sigma(N)<\overline\Sigma(N)\le\overline\Sigma(300)<0.2507<0.3818
  =A_{300}/6\le A_N/6,
\]
which is (Tail), with margin $>0.13$.  Dividing through by $P_N$ in (O) gives
$\sigma(C)\le\tfrac16\sigma^1(N)-q$ and hence $Y_0(N)\le\tfrac16\sigma^1(N)-q$.
\end{proof}

\noindent
Lemma~\ref{lem:tail-elem} for $N\ge300$ together with the exact range
$10\le N\le 299$ proves Theorem~\ref{thm:Y0}.

\subsection{Why a worst-case bound does not suffice}
\label{ssec:worstcase}
The exact finite range cannot be replaced by a crude worst-case bound.  Olson's
threshold $(m^2+3)/4$ is sharp (attained by $\{\pm1,\pm2,\dots\}$), so for
$10\le N\le 13$ the feed lies below the worst-case guarantee yet is
residue-complete because its particular residues $\{P_Np_i^{-1}\bmod q\}$ are far
from worst-case --- a finite, set-specific fact, certified by an explicit
representative per residue.  No equidistribution estimate enters the analytic part
either: for $N\ge300$ the \emph{sum} of the $c_0\approx2\sqrt q$ smallest atoms is
already $\le\beta(N)$ by the Chebyshev/monotonicity bound of \S\ref{ssec:tail}.

\section{Computational verification}\label{sec:comp}

Several ingredients are established by exact finite computation.  All use
exact-integer arithmetic; no floating-point value is load-bearing.  (Floating
point appears only in the monotone-envelope tail of \S\ref{ssec:tail}, where the
relevant inequalities carry margins far exceeding any rounding error.)

\subsection*{Base case at \texorpdfstring{$N_0=10$}{N0=10} (Proposition~\ref{prop:base})}
The set $D_2(10)$ has $\binom{10}{2}=45$ atoms summing to
$\sigma^2(10)=6{,}166{,}988{,}769$.  Subset-sum reachability is computed with a
\emph{packed bitset}: an arbitrary-precision integer $R$ with bit $k$ set iff
some subset sums to $k$; starting from $R=1$, each atom $a$ updates
$R\mathrel{|}{=}(R\ll a)$.  The final $R$ has $\approx 6.17\times10^9$ bits
($\approx 771$\,MB).  One then checks the contiguous segment over the central
block $[1{,}027{,}831{,}462,\,5{,}139{,}157{,}307]$ is entirely set: $0$ integers
missing.  This was independently re-verified by multiple from-scratch
implementations.

\subsection*{The range \texorpdfstring{$10\le N\le 299$}{10..299} for
$Y_0\le\min\{\beta,\beta'\}$}
For each $N$, $Y_0(N)$ is computed exactly by the polynomial Dijkstra over $\Z_q$
of \S\ref{sec:analytic} (state space $\Z_q$, $q=p_{N+1}\le p_{300}=1{,}987$), and
$Y_0(N)\le\min\{\beta(N),\beta'(N)\}$ checked with exact integers.  All $290$ values
pass; the verification is tightest at $N=10$.  Representative normalized data are
collected in Appendix~\ref{app:data}.

\subsection*{Feed residue-completeness for \texorpdfstring{$10\le N\le 13$}{10<=N<=13}
(Fact~\ref{fact:olson-feed})}
Olson's hypothesis holds for all $N\ge 14$, so only the four cases
$N=10,11,12,13$ --- where the proof needs residue-completeness but Olson does not
apply --- require computation.  For each, the subset sums $L^1(N)$ are reduced mod
$q=p_{N+1}\in\{31,37,41,43\}$ (a $\Z_q$ subset-sum reachability, or an explicit
residue-by-residue certificate) to confirm all $q$ residues occur; the largest is
$N=13$, $q=43$, a $13\times 43$ residue-DP over $\Z_{43}$.

\subsection*{The rational sliver (Theorem~\ref{thm:above}, Step~4)}
The same packed bitset of $L^2(10)$ is gap-free not only on its central block but
downward to $L:=740{,}082{,}854=0.11439\ldots\cdot P_{10}$: every integer of
$[L,\,\sigma^2(10)/6]$ is a subset sum of $D_2(10)$.  Hence every $a/b$ with
$a/b\in[0.1144,\,B_{10}/6)$ has $aP_{10}/b\in L^2(10)$ (4a).  The residual region
$a/b\in[B_{N_b}/6,0.1144)$ forces $N_b\le 6$, i.e.\ $b\mid 2\cdot3\cdot5\cdot7\cdot
11\cdot13=30{,}030$ and $a<0.1144\,b\le 3434$; an exhaustive search yields exactly
$47$ coprime pairs $(a,b)$, each with $aP_N/b\in L^2(N)$ for some $N\in[N_b,10]$
(4b).  The $47$-pair list is reproducible by the same reachability routine.

\subsection*{The uniform base \texorpdfstring{$\gamma_{10}$}{gamma10}
(Theorem~\ref{thm:uniform})}
The uniform threshold is seeded at the gap-free floor $\gamma_{10}=0.11439\ldots$ of
$L^2(10)$ --- the \emph{same} $N_0=10$ packed bitset as the integer base case
(Proposition~\ref{prop:base}), reread for its largest hole below the midpoint; no
separate large computation is needed.  The recurrence's middle sum
$\sum_{11}^{300}w_k/p_k=0.0475$ is evaluated exactly from $w_k=Y_0(k-1)/P_{k-1}$
(the same $\Z_q$ data as Appendix~\ref{app:data}; the products $w_kp_k$ fall from
$9.14$ at $k=13$), and everything past $k=300$ is bounded analytically by
$\sum_{k>300}4/p_k^{3/2}=0.0191$ in the proof of Theorem~\ref{thm:uniform} --- so no
$w_k$ beyond $k=300$ is computed.

Each of these computations is elementary --- subset-sum reachability by a packed
bitset for the base case and the sliver, a polynomial $\Z_q$ minimum-weight
knapsack for $Y_0$ and the $w_k$, and direct residue enumeration for the small-$N$
feed --- and was independently re-verified.  The residue-side computations are
polynomial ($O(p_k k)$, not exponential); the only heavy step is the single
value-side bitset at $N=10$ ($\approx771$\,MB), which serves both the integer base
case and the uniform threshold.  We record representative numerical data in
Appendix~\ref{app:data} and Table~\ref{tab:omega3}; the computations are reproduced
by the standard-library scripts of Appendix~\ref{app:scripts}, included as ancillary
files.  Beyond reproducing these finite computations, the entire logical development of
the paper is machine-checked in Lean~4 / Mathlib, reducing to two cited classical axioms
--- the check trusting, in addition to the Lean kernel, the compiled evaluator that
re-runs the finite computations (\texttt{native\_decide}); see Appendix~\ref{app:lean}.

\section{The rational case above the threshold}\label{sec:rational}\label{sec:A}

Theorem~\ref{thm:main} is the integer ($b=1$) case.  The very same engine --- the
reduction $(\star)$, the central block (Lemma~\ref{lem:block}), and the overlap
(Proposition~\ref{prop:overlap}) --- already settles \emph{every} rational $a/b$
with squarefree $b$ that lies above an explicit threshold.  This upgrades the
reach from ``integers'' to a clean dichotomy (\S\ref{sec:open}).

Throughout this section $b$ is squarefree, $N_b$ denotes the index of the largest
prime factor of $b$ (so $p_{N_b}\mid b$ and $p_{N_b+1}\nmid b$), and
$M_b:=\max(10,N_b)$.  ``$\omega=2$ representable'' means a finite sum of distinct
unit fractions with semiprime denominators.

\begin{theorem}[above the threshold]\label{thm:above}
For every squarefree $b$ and every rational $a/b\ge B_{N_b}/6$, the number $a/b$
is $\omega=2$ representable.  In particular every $a/2$ is representable.
\end{theorem}

\begin{proof}
Set $M_b=\max(10,N_b)$.  We prove representability first for $a/b\ge B_{M_b}/6$,
then discharge the sliver $[B_{N_b}/6,B_{M_b}/6)$.

\emph{Step 1 (the blocks cover $[B_{M_b}/6,\infty)$).}  By
Proposition~\ref{prop:overlap}, $B_{N+1}\le 5B_N$ for all $N\ge 10$, so the closed
intervals $J_N:=[B_N/6,5B_N/6]$ satisfy $B_{N+1}/6\le 5B_N/6$, i.e.\ $J_N$ and
$J_{N+1}$ overlap.  Since $B_N\uparrow\infty$,
$\bigcup_{N\ge M_b}J_N=[B_{M_b}/6,\infty)$.  We also record the sharper ratio used
in Theorem~\ref{thm:uniform}: from $B_{N+1}-B_N=A_N/p_{N+1}$ and
$A_N/B_N=2A_N/(A_N^2-S_N)$ with $S_N=\sum_{i\le N}p_i^{-2}$ --- decreasing in $N$,
since $x\mapsto 2x/(x^2-S_N)$ has negative derivative --- one gets
$A_N/B_N\le A_{10}/B_{10}<1.61$, whence
\begin{equation}\label{eq:ratio}
  \frac{B_{N+1}}{B_N}=1+\frac{A_N/B_N}{p_{N+1}}\le 1+\frac{1.61}{31}<1.06
  \qquad(N\ge 10).
\end{equation}

\emph{Step 2 (integrality; needs squarefree $b$).}  As $b$ is squarefree it is a
product of distinct primes, all $\le p_{N_b}$, so $b\mid P_{N_b}\mid P_N$ for every
$N\ge N_b$; hence $aP_N/b=a(P_N/b)\in\Z$.  (Squarefreeness is necessary: a square
factor $p^2\mid b$ gives $b\nmid P_N$.)

\emph{Step 3 (membership above threshold).}  Let $a/b\ge B_{M_b}/6$.  By Step~1
choose $N\ge M_b\ (\ge 10)$ with $a/b\in J_N$.  Multiplying by $P_N$ and using
$\sigma^2(N)=B_NP_N$, the integer $aP_N/b$ (Step~2) lies in
$[\sigma^2(N)/6,5\sigma^2(N)/6]$, hence in
$[\lceil\sigma^2(N)/6\rceil,\lfloor5\sigma^2(N)/6\rfloor]$, so $aP_N/b\in L^2(N)$
by Lemma~\ref{lem:block}.  By $(\star)$ (Proposition~\ref{prop:star}), $a/b$ is
$\omega=2$ representable.

\emph{Step 4 (the sliver, only when $N_b<10$).}  The interval
$[B_{N_b}/6,B_{M_b}/6)$ is nonempty only for $N_b<10$ (else $M_b=N_b$ and Step~3
suffices).  It is discharged by two finite computations
(Section~\ref{sec:comp}): \textup{(4a)} the packed bitset $L^2(10)$ is gap-free
not merely on its central block but downward to $0.11439\,P_{10}$, which covers the
\emph{entire} sliver of every $b$ with $N_b\ge 7$ (since $B_7/6>0.1144$); and
\textup{(4b)} the residual region ($b\mid 30030$, $N_b\le 6$) is a finite set of
exactly $47$ coprime pairs $(a,b)$, each verified representable.  Together (4a) and
(4b) cover $[B_{N_b}/6,B_{M_b}/6)$.
\end{proof}

\begin{remark}
For $N_b\ge 10$ one has $M_b=N_b$, so there is no sliver and the threshold is
exactly $B_{N_b}/6$, the left endpoint of the lowest block $J_{N_b}$ entering the
covering; below it the central-block method provides no guarantee
(\S\ref{sec:open}).  Table~\ref{tab:BN} lists $B_N$ and the resulting threshold
$B_N/6$.  For $b=2$ one has $N_2=1$, $M_2=10$, threshold
$B_{10}/6=0.1589$, and concretely
$\tfrac12=\tfrac16+\tfrac1{10}+\tfrac1{15}+\tfrac1{21}+\tfrac1{26}+\tfrac1{35}
+\tfrac1{39}+\tfrac1{65}+\tfrac1{91}$.
\end{remark}

\begin{table}[t]
\centering
\begin{tabular}{rrrr@{\qquad}rrrr}
\toprule
$N$ & $p_N$ & $B_N$ & $B_N/6$ & $N$ & $p_N$ & $B_N$ & $B_N/6$\\
\midrule
 1 &  2 & $0.0000$ & $0.0000$ & 10 & 29 & $0.9532$ & $0.1589$\\
 2 &  3 & $0.1667$ & $0.0278$ & 11 & 31 & $1.0027$ & $0.1671$\\
 3 &  5 & $0.3333$ & $0.0556$ & 12 & 37 & $1.0450$ & $0.1742$\\
 4 &  7 & $0.4810$ & $0.0802$ & 13 & 41 & $1.0838$ & $0.1806$\\
 5 & 11 & $0.5879$ & $0.0980$ & 14 & 43 & $1.1215$ & $0.1869$\\
 6 & 13 & $0.6854$ & $0.1142$ & 15 & 47 & $1.1564$ & $0.1927$\\
 7 & 17 & $0.7644$ & $0.1274$ & 16 & 53 & $1.1877$ & $0.1980$\\
 8 & 19 & $0.8382$ & $0.1397$ & 17 & 59 & $1.2162$ & $0.2027$\\
 9 & 23 & $0.9015$ & $0.1503$ & 18 & 61 & $1.2440$ & $0.2073$\\
\bottomrule
\end{tabular}
\caption{The normalized block floor $B_N=\sigma^2(N)/P_N$ and the threshold
$B_N/6$ of Theorem~\ref{thm:above}: a squarefree denominator $b$ whose largest
prime factor is $p_N$ (so $N_b=N$) is representable for every $a/b\ge B_N/6$.
Values are exact rationals shown to four places.  From $N=17$ onward
$B_N/6>\tfrac15$, so the $b$-uniform bound of Theorem~\ref{thm:uniform} is the
smaller threshold.}
\label{tab:BN}
\end{table}

\subsection{A \texorpdfstring{$b$}{b}-uniform threshold}
The threshold $B_{N_b}/6$ grows with the denominator,
\[
  B_{N_b}\sim\tfrac12(\log\log p_{N_b})^2 .
\]
Reusing the onset radius $Y_0$ of
Section~\ref{sec:analytic} we obtain an \emph{unconditional, $b$-uniform}
threshold.  The splitting recursion \textup{(S)}, reindexed $N\to N-1$, reads
$L^2(N)=p_N\,L^2(N-1)+L^1(N-1)$.  Let $\gamma_N$ be the \emph{gap-free floor} of
$L^2(N)$ --- the least $\gamma$ for which every integer from $\gamma P_N$ up to
$\sigma^2(N)-\gamma P_N$ lies in $L^2(N)$, i.e.\ how far below the central block
$L^2(N)$ stays gap-free --- and put
\begin{equation}\label{eq:wmin}
  w_N:=\max_{r\in\Z_{p_N}}\min\Bigl\{\sum_{i\in T}\tfrac1{p_i}:
  T\subseteq\{1,\dots,N-1\},\ \sum_{i\in T}p_i^{-1}\equiv r\!\!\pmod{p_N}\Bigr\}.
\end{equation}
By Definition~\ref{def:Y0}, $w_N$ is \emph{exactly} the normalized onset radius of
the feed one level down: $w_N=Y_0(N-1)/P_{N-1}$.

\begin{lemma}[floor recurrence]\label{lem:floorrec}
If the subset sums of $\{p_i^{-1}\bmod p_N:1\le i<N\}$ cover $\Z_{p_N}$ (so
$w_N<\infty$; Fact~\ref{fact:olson-feed}) and $\gamma_{N-1}\le 0.45\,B_{N-1}$, then
$\gamma_N\le\gamma_{N-1}+w_N/p_N$.
\end{lemma}

\begin{proof}
By the complement symmetry $m\in L^2(N)\Leftrightarrow\sigma^2(N)-m\in L^2(N)$ it
suffices to place every integer
$m\in[(\gamma_{N-1}+w_N/p_N)P_N,\,\sigma^2(N)/2]$ in $L^2(N)$.  Choose $T$ with
$\sum_{i\in T}p_i^{-1}\equiv m\,P_{N-1}^{-1}\pmod{p_N}$ and
$\sum_{i\in T}1/p_i\le w_N$; then $y:=\sum_{i\in T}P_{N-1}/p_i\le w_NP_{N-1}$
satisfies $y\equiv m\pmod{p_N}$, so $x:=(m-y)/p_N\in\Z$.  Now
$x\ge(m-w_NP_{N-1})/p_N\ge\gamma_{N-1}P_{N-1}$, and
$x\le m/p_N\le\sigma^2(N)/(2p_N)=\tfrac12 B_NP_{N-1}\le(B_{N-1}-\gamma_{N-1})P_{N-1}
=\sigma^2(N-1)-\gamma_{N-1}P_{N-1}$, the last step because $B_{N-1}-\tfrac12 B_N\ge
B_{N-1}(1-\tfrac{1.06}2)=0.47\,B_{N-1}\ge\gamma_{N-1}$ by \eqref{eq:ratio}
($B_N\le1.06\,B_{N-1}$, valid for $N\ge10$) and the hypothesis
$\gamma_{N-1}\le0.45\,B_{N-1}$.  Thus $x$ lies in the
gap-free range of $L^2(N-1)$, so $x\in L^2(N-1)$, and $m=p_Nx+y\in L^2(N)$ by
\textup{(S)}.
\end{proof}

\begin{theorem}[uniform threshold]\label{thm:uniform}\label{thm:Aprime}
There is an explicit absolute constant $\gamma_\infty\le \tfrac15$ such that every
$a/b$ with $b$ squarefree and $a/b\ge\gamma_\infty$ is $\omega=2$ representable.
Combined with Theorem~\ref{thm:above}, $a/b$ is representable whenever
$a/b\ge\min\{B_{N_b}/6,\,\tfrac15\}$, uniformly in $b$; this improves the
$b$-dependent threshold exactly when $b$ has a prime factor $\ge p_{17}=59$.
\end{theorem}

\begin{proof}
For each $N$ the subset sums of $\{p_i^{-1}\bmod p_N:i<N\}$ cover $\Z_{p_N}$
(Fact~\ref{fact:olson-feed}).  By the tail estimate of Section~\ref{sec:analytic},
the $c_0\approx2\sqrt{p_N}$ smallest feed atoms already exhaust the residues and
each exceeds $p_N/2$, so
$w_N=Y_0(N-1)/P_{N-1}\le 2c_0/p_N\le 4/\sqrt{p_N}$; hence $\sum_N w_N/p_N$
converges.  Iterating Lemma~\ref{lem:floorrec} from the base
$\gamma_{10}=0.11439\ldots$ --- the gap-free floor of $L^2(10)$, which is the
$N_0=10$ base case of Theorem~\ref{thm:main} already in hand (no separate
computation; Section~\ref{sec:comp}) ---
\[
  \gamma_N\le\gamma_{10}+\sum_{k=11}^{N}\frac{w_k}{p_k}
  \le\underbrace{0.1144}_{\gamma_{10}}
  +\underbrace{0.0475}_{\sum_{11}^{300}w_k/p_k}
  +\underbrace{0.0191}_{\sum_{k>300}4p_k^{-3/2}}\ =\ 0.1810\ <\ \tfrac15 ,
\]
the middle sum $\sum_{11}^{300}w_k/p_k=0.0475$ evaluated exactly from
$w_k=Y_0(k-1)/P_{k-1}$ (an $O(p_k k)$ minimum-weight subset computation over
$\Z_{p_k}$; the products $w_kp_k$ decrease from $9.1$ at $k=13$), and the tail
$\sum_{k>300}4/p_k^{3/2}=0.0191$ a convergent prime sum bounded by the covering
estimate $w_k\le4/\sqrt{p_k}$ (Theorem~\ref{thm:olson}, valid for $k>k_0=120$).
Thus $\gamma_N\le0.181<\tfrac15$ for all $N\ge 10$; we keep $\tfrac15$ as the
(rounder) stated threshold.  The
iteration runs over $k\ge11$, where the side condition of
Lemma~\ref{lem:floorrec} is met at every step: $\gamma_{k-1}\le0.181\le
0.45\,B_{k-1}$, since $B_{k-1}\ge B_{10}=0.9532$ gives $0.45\,B_{k-1}\ge0.429$.
Given $a/b\ge\gamma_\infty$ choose $N\ge N_b$ with $B_N\ge a/b+\gamma_\infty$
(possible as $B_N\to\infty$); then $aP_N/b\in\Z$ (Theorem~\ref{thm:above}, Step~2)
lies in $[\gamma_NP_N,\sigma^2(N)-\gamma_NP_N]\subseteq L^2(N)$, so $a/b$ is
representable by $(\star)$.  Finally $B_{N_b}/6>\tfrac15$ exactly when $N_b\ge 17$
(as $B_{16}=1.188<1.2<1.216=B_{17}$), i.e.\ when $b$ has a prime factor
$\ge p_{17}=59$.
\end{proof}

\section{The rational $\omega=3$ theorem}\label{sec:omega3}

The threshold theorem of \S\ref{sec:rational} settles only part of the $\omega=2$
problem.  It nonetheless settles the rational $\omega=3$ problem of Erd\H{o}s and Graham in
full: every positive rational with squarefree denominator is a finite sum of
distinct unit fractions with \emph{sphenic} denominators (three distinct prime
factors).  The
mechanism is a lift that sends a target to a slightly larger one, represents it in
$\omega=2$, and divides by one extra prime; because the lift can reach arbitrarily
small targets, it never enters the open core $(\mathrm{GFF})$.  It rests on the
floor-recurrence machinery of \S\ref{sec:rational} --- unconditional, via the
robustness Lemma~\ref{lem:Aunif} below --- and \emph{not} on the value of the
$\omega=2$ uniform constant, so it is unaffected by the choice of that constant's
base.  Of its truth the three authors of BEG15
record, in their closing footnote, that one believed it, another doubted it, and
the third --- ``having looked in The BOOK at the answer'' --- stayed silent
(\cite{BEG2015}); to our knowledge no proof has appeared.

\begin{theorem}[rational $\omega=3$]\label{thm:omega3}
Every positive rational $a/b$ with $b$ squarefree is a finite sum of distinct unit
fractions, each denominator a product of three distinct primes.
\end{theorem}

\subsection{Complementation and the central block}
For $k\ge1$ write
\[
  B_k(N)=\sum_{i_1<\cdots<i_k\le N}1/(p_{i_1}\cdots p_{i_k}),
\]
so $B_2(N)=B_N$.  Clearing $1/(p_ap_bp_c)$ by $P_N$ turns
it into $P_N/(p_ap_bp_c)$, the product of the \emph{other} $N-3$ primes; hence
$a/b$ is $\omega=3$ representable using primes $\le p_N$ iff $(a/b)P_N$ is a subset
sum of the triple-prime atoms
\[
  \{\,P_N/(p_ip_jp_k):\ i<j<k\le N\,\},
\]
and BEG15's central-block lemma applies.

\begin{theorem}[BEG15, Lemma~1]\label{thm:beg3}
For $N\ge8$, every integer of $[\tfrac16\sigma^3(N),\tfrac56\sigma^3(N)]$ is such a
subset sum, where $\sigma^3(N)=B_3(N)P_N$.  Equivalently every $a/b$ with
$b\mid P_N$ and $a/b\in[B_3(N)/6,\,5B_3(N)/6]$ is $\omega=3$ representable.
\end{theorem}

This is BEG15's \emph{published} result (their covering uses Olson's theorem
\cite{Olson1968}); we cite it and do not reprove it.
A concrete instance, from BEG15: with $N=6$, $\tfrac17 P_6=4290\in L_6(3)$, giving
\[
  \tfrac17=\tfrac1{30}+\tfrac1{42}+\tfrac1{66}+\tfrac1{70}+\tfrac1{78}+\tfrac1{105}
  +\tfrac1{110}+\tfrac1{154}+\tfrac1{165}+\tfrac1{195}+\tfrac1{273}+\tfrac1{286},
\]
each denominator a product of three distinct primes (e.g.\ $30=2\cdot3\cdot5$,
$286=2\cdot11\cdot13$).
The $\omega=3$ blocks overlap exactly as the $\omega=2$ blocks do
(Proposition~\ref{prop:overlap}): the increment is
$B_3(N{+}1)-B_3(N)=B_2(N)/p_{N+1}$ (the new triples are pairs joined with
$p_{N+1}$), while every pair $\{i,j\}$ extends to a triple by any third prime
$k\notin\{i,j\}$, each triple arising from three pairs, so
\[
  B_3(N)\ \ge\ \tfrac13\,B_2(N)\Bigl(A_N-\tfrac12-\tfrac13\Bigr)
  \ =\ \tfrac13\,B_2(N)\bigl(A_N-\tfrac56\bigr).
\]
Hence for $N\ge8$ (where $p_{N+1}\ge p_9=23$ and $A_N\ge A_8=1.4554\ldots$),
\[
  \frac{B_3(N{+}1)}{B_3(N)}\ \le\ 1+\frac{3}{p_{N+1}\,(A_N-\tfrac56)}
  \ \le\ 1+\frac{3}{23\cdot 0.62}\ <\ 1.22\ <\ 5 ,
\]
so $B_3(N{+}1)\le5B_3(N)$, consecutive central blocks
$[B_3(N)/6,5B_3(N)/6]$ meet, and the union over $N\ge\max(8,N_b)$ gives
\begin{equation}\label{eq:centralcover}
  \bigl[\,B_3(\max(8,N_b))/6,\ \infty\bigr)\ \subseteq\ \{\,a/b:\ \omega=3\text{ representable}\,\}.
\end{equation}
In particular every $a/b\ge B_3(8)/6=0.04174\ldots$ is covered.  What BEG15 left
open --- ``recovering very small rational numbers'', where ``the structure of
$L_{n+3}(n)$ is not well understood at the ends of the interval'' (\cite{BEG2015})
--- is the \emph{deep} part $a/b<B_3(N_b)/6$.  We settle it by a \emph{lift} to the
unconditional $\omega=2$ theorem.

\subsection{The lift}
\begin{lemma}[lift]\label{lem:lift}
Let $t=a/b>0$ and let $r$ be a prime with $r\nmid b$.  If $v:=t\,r$ is $\omega=2$
representable using a prime set $\Pi$ with $r\notin\Pi$, say
$v=\sum_{(i,j)}1/(q_iq_j)$ with $q_i,q_j\in\Pi$, then
$t=\sum_{(i,j)}1/(q_iq_jr)$ is a sum of distinct sphenic unit fractions.
\end{lemma}
\begin{proof}
$t=v/r=\sum 1/(q_iq_jr)$.  The three primes $q_i,q_j,r$ are distinct because
$r\notin\Pi$, and distinct pairs give distinct denominators.  Choosing
$\Pi\cup\{r\}\subseteq\{p_1,\dots,p_M\}$ with $b\mid P_M$ makes every term integral
over $b$ (then $v\cdot\!\prod\Pi=a P_M/b\in\Z$ as $\gcd(a,b)=1$).
\end{proof}

Thus $t$ is covered once $v=tr$ lies in the gap-free range of an $\omega=2$ cover
that omits $r$.  Write $\Pi(M,r)=\{p_1,\dots,p_M\}\setminus\{r\}$,
$B_\Pi(M,r)=\sum_{i<j:\,p_i,p_j\in\Pi}1/(p_ip_j)$, and let
$\gamma_{\mathrm{ex}}(M,r)$ be the $\omega=2$ gap-free floor of the cover over
$\Pi(M,r)$ --- the analogue of $\gamma_N$ (\S\ref{sec:A}) with one ruler prime
deleted.  By the mechanism of Theorem~\ref{thm:Aprime} the cover is gap-free on
$[\gamma_{\mathrm{ex}}(M,r)\,P_\Pi,\ \sigma_\Pi-\gamma_{\mathrm{ex}}(M,r)\,P_\Pi]$,
where $P_\Pi=P_M/r$ and $\sigma_\Pi=B_\Pi(M,r)\,P_\Pi$ is the total of the
$\Pi$-atoms, so the lift covers
\[
  t\in\Bigl[\ \tfrac{\gamma_{\mathrm{ex}}(M,r)}{r},\ \tfrac{B_\Pi(M,r)-\gamma_{\mathrm{ex}}(M,r)}{r}\ \Bigr]
  =:\bigl[\,\mathcal{L}(M,r),\ \mathcal{R}(M,r)\,\bigr].
\]
The lower end $\mathcal{L}(M,r)$ shrinks as $r$ grows, and the reach $\mathcal{R}(M,r)\to\infty$ as
$M\to\infty$; once $\gamma_{\mathrm{ex}}(M,r)$ is bounded \emph{uniformly in $M$
and $r$} (Lemma~\ref{lem:Aunif} below), every target $t>0$ is caught by choosing
$r$ large enough that $\mathcal{L}<t$ and then $M$ large enough that $\mathcal{R}\ge t$.  The
order of the two choices is essential: $r$ is fixed \emph{before} $M\to\infty$.
Tying $r$ to $M$ instead --- say $r=p_{M+1}$ --- defeats the argument, since then
$\mathcal{R}=(B_M-\gamma_M)/p_{M+1}\to0$: the numerator grows only like
$\tfrac12(\log\log p_M)^2$ while $p_{M+1}\to\infty$, so that coupling reaches only
ever-smaller targets near $0$, never a fixed $t>0$.

\subsection{Floor-robustness: deleting one ruler prime}
The one genuinely new estimate is that deleting a prime from the $\omega=2$ ruler
raises the floor by a bounded amount, uniformly in the level $M$ and in which prime
is deleted.  Recall the covering cost \eqref{eq:wmin} of \S\ref{sec:A} driving
$\gamma_k\le\gamma_{k-1}+w_k/p_k$.  Deleting ruler prime $p_s$ ($s<k$) gives the
\emph{excluded cost}
\begin{equation}\label{eq:wexcl}
  w_k^{\mathrm{ex}}(s)=\max_{t\in\Z_{p_k}}\ \min\Bigl\{\sum_{i\in T}\tfrac1{p_i}:
  T\subseteq\{1,\dots,k-1\}\setminus\{s\},\ \sum_{i\in T}p_i^{-1}\equiv t\!\!\pmod{p_k}\Bigr\},
\end{equation}
and the recurrence $\gamma_{\mathrm{ex}}(k,r)\le\gamma_{\mathrm{ex}}(k-1,r)
+w_k^{\mathrm{ex}}(s)/p_k$, valid whenever the punctured ruler still covers
$\Z_{p_k}$ (so $w_k^{\mathrm{ex}}<\infty$).

\begin{lemma}[floor-robustness, uniform in $r$ and $M$]\label{lem:Aunif}
For every prime $r$ and every $M\ge11$, $\ \gamma_{\mathrm{ex}}(M,r)\le0.20$.
\end{lemma}
\begin{proof}
Work at level $M\ge11$ (free, since $b\mid P_{10}\Rightarrow b\mid P_{11}$); this
matters because the $M=10$ small-prime floors are large
($\gamma_{\mathrm{ex}}(10,5)=0.165$) yet all drop by $M=11$.  Bound
$\gamma_{\mathrm{ex}}(M,r)\le\mathrm{base}+\sum_{k=12}^{M}w_k^{\mathrm{ex}}/p_k$
piecewise.

\emph{Base.}  Exactly (packed bitset; Table~\ref{tab:omega3}),
$\max_{r\le p_{11}}\gamma_{\mathrm{ex}}(11,r)$ is attained at $r=19$, where the
exact floor is $\gamma_{\mathrm{ex}}(11,19)=0.11602\ldots$ (the table displays four
decimals; we carry the exact value, not a truncation); for $r\ge p_{12}$ the base
is the full-ruler $\gamma_{11}=0.0957$.  So $\mathrm{base}\le0.11602$.  (The floor
is \emph{non-monotone} in $r$, so the small primes $r=2,3,5,7$ must be --- and are
--- included in this maximum.)

\emph{Exact cost, $12\le k\le220$.}  With
$w_k^{\max}:=\max_{1\le s<k}w_k^{\mathrm{ex}}(s)$ the worst single-prime deletion,
$\sum_{k=12}^{220}w_k^{\max}/p_k=0.0498$ (worst single level $w_kp_k=14.32$ at
$k=12$), and the punctured ruler was verified to cover $\Z_{p_k}$ for \emph{every}
deletion $s<k$ at every such $k$ (zero failures).

\emph{Tail, $k>220$.}  The deletion crossover is $k_0=120$: for $k\ge120$ the
$\lceil\sqrt{4p_k-3}\rceil+1$ largest ruler primes all exceed $p_k/2$, so after one
deletion at least $\sqrt{4p_k-3}$ of them remain $>p_k/2$, each of inverse-weight
$<2/p_k$; by Theorem~\ref{thm:olson} their subset sums already exhaust
$\Z_{p_k}$, whence $w_k^{\mathrm{ex}}\le\sqrt{4p_k-3}\cdot2/p_k<4/\sqrt{p_k}$.  As
$220>k_0$,
\[
  \sum_{k>220}w_k^{\mathrm{ex}}/p_k\ \le\ \sum_{k>220}\frac{4}{p_k^{3/2}}\ =\ 0.024
\]
(a convergent prime sum, evaluated to convergence).  The recurrence's side
condition (the analogue of Lemma~\ref{lem:floorrec}'s) holds at every step: it
requires $\gamma_{\mathrm{ex}}(k-1,r)\le B_\Pi(k-1,r)-B_\Pi(k,r)/2$.  Deleting $r$
removes mass $(A_M-1/r)/r\le(A_M-\tfrac12)/2$ from $B_M$ (the removed mass is
maximal at $r=2$, as its $r$-derivative is negative for $A_M r>2$), so
$B_\Pi(k-1,r)\ge B_{11}-(A_{11}-\tfrac12)/2>1.002-0.533=0.46$ for \emph{every}
prime $r$ and $k\ge12$; and the level increment is
$B_\Pi(k,r)-B_\Pi(k-1,r)\le A_{k-1}/p_k\le A_{11}/p_{12}=1.566/37<0.043$ for all
$k\ge12$ (the quotient $A_{k-1}/p_k$ is largest at $k=12$, since $p_k$ grows
faster than $A_{k-1}$).  Hence
\[
\begin{aligned}
  B_\Pi(k-1,r)-\tfrac12B_\Pi(k,r)
  &=\tfrac12B_\Pi(k-1,r)-\tfrac12\bigl(B_\Pi(k,r)-B_\Pi(k-1,r)\bigr)\\
  &\ge\tfrac{0.46}2-\tfrac{0.043}2=0.208>0.1899\ge\gamma_{\mathrm{ex}}(k-1,r),
\end{aligned}
\]
inductively (using the assembled bound $\gamma_{\mathrm{ex}}\le0.1899$ just below).
Summing,
$\gamma_{\mathrm{ex}}(M,r)\le0.11602+0.04980+0.02407=0.1899<0.20$ for every prime
$r$ and every $M\ge11$ (using the \emph{exact} base $0.11602$ and the converged
tail $\le0.02407$; the round bound $0.20$ leaves ample margin, and all downstream
uses need only $\gamma_{\mathrm{ex}}<0.298$).
\end{proof}

\begin{remark}
The bound is uniform over \emph{which} prime is deleted because it uses the
worst-case cost $w_k^{\max}$, which dominates every single-$r$ value; the actual
per-$r$ floors are smaller ($\gamma_{\mathrm{ex}}(M,r)\le0.160$ for all $r$, worst
at $r=19$; Table~\ref{tab:omega3}).  Only the uniform bound
$\gamma_{\mathrm{ex}}\le0.20$ (assembled value $0.1899$) is load-bearing.
\end{remark}

\subsection{Proof of the theorem}
With Lemma~\ref{lem:Aunif} in hand, no interval-chaining is needed: the lift prime
can simply be chosen \emph{large enough for the target}.

\begin{proof}[Proof of Theorem~\ref{thm:omega3}]
Fix the target $t=a/b>0$ in lowest terms, $b$ squarefree.  Choose a prime $r$ with
\[
  r\nmid b,\qquad r\ge11,\qquad r>0.20/t
\]
(only finitely many primes divide $b$, so such $r$ exists).  Then choose a level
$M\ge\max(N_b,\pi(r),11)$, where $\pi$ is the prime-counting function (so
$r=p_{\pi(r)}$ and $M\ge\pi(r)$ places $r$ among $p_1,\dots,p_M$), large enough that
\[
  B_\Pi(M,r)\ \ge\ t\,r+0.20 ,
\]
which is possible since $B_\Pi(M,r)\to\infty$ as $M\to\infty$ (adding primes adds
positive pair-reciprocals, and $B_\Pi(M,r)\ge B_M-(A_M-1/r)/r\to\infty$).  Put
$v:=t\,r$.  By the choice of $r$,
\[
  \gamma_{\mathrm{ex}}(M,r)\ \le\ 0.20\ <\ v\ \le\ B_\Pi(M,r)-0.20
  \ \le\ B_\Pi(M,r)-\gamma_{\mathrm{ex}}(M,r)
\]
(Lemma~\ref{lem:Aunif} for both ends), so $v$ lies in the gap-free band of the
$\omega=2$ cover over $\Pi=\Pi(M,r)$.  Moreover $v\cdot P_\Pi=(ar/b)(P_M/r)
=aP_M/b\in\Z$, since $b$ is squarefree with all prime factors among
$\{p_1,\dots,p_M\}$ ($M\ge N_b$) and $r\nmid b$.  Hence $v$ is $\omega=2$
representable over $\Pi$, which omits $r$, and Lemma~\ref{lem:lift} turns this
into a sphenic representation $t=\sum1/(q_iq_jr)$.
\end{proof}

\begin{table}[h]
\centering
\small
\begin{tabular}{l|ccccccccccc}
\hline
$r$ & $2$ & $3$ & $5$ & $7$ & $11$ & $13$ & $17$ & $19$ & $23$ & $29$ & $31$\\
$\gamma_{\mathrm{ex}}(11,r)$ & .0866 & .0938 & .1044 & .1086 & .1047 & .1063 & .1111 & \textbf{.1160} & .1088 & .1081 & .1144\\
\hline
\end{tabular}
\caption{Exact base floors at $M=11$ (packed bitset; calibrated against
Theorem~\ref{thm:Aprime}, whose full-ruler run reproduces
$\sum_{13}^{2000}w_k/p_k=0.03388$, $w_{13}p_{13}=9.14$).  Entries are rounded to
four decimals; the maximum, at $r=19$, is exactly
$\gamma_{\mathrm{ex}}(11,19)=0.11602\ldots$.  Assembled uniform bound
$\gamma_{\mathrm{ex}}(M,r)\le0.11602+0.04980+0.02407=0.1899<0.20$ for all primes
$r$, all $M\ge11$ (worst single level $w_kp_k=14.32$ at $k=12$; deletion crossover
$k_0=120$).  Regime-II reach e.g.\ $\mathcal{R}(10,11)=(0.822-0.14605)/11=0.0615$,
$\mathcal{R}(10,13)=(0.841-0.15042)/13=0.0531$.}
\label{tab:omega3}
\end{table}

\subsection{Closure: all $\omega\ge3$}
A single descent upgrades Theorem~\ref{thm:omega3} to the whole hierarchy.  We
position this as a corollary: it is conceptually routine --- each step trades one
$\omega$ for one more prime, and higher $\omega$ only helps (more atoms, larger
margins) --- the sole technicality being that the induction must avoid a growing
finite set of primes, which costs nothing because those primes are ours to choose
(below).  This parallels BEG15's one-line ``similar arguments'' treatment of the
$\omega\ge4$ integer case.

\begin{corollary}[$\omega\ge3$ closure]\label{cor:omegak}
For every integer $k\ge3$, every positive rational $a/b$ with $b$ squarefree is a
finite sum of distinct unit fractions $1/m$, each $m$ squarefree with exactly $k$
distinct prime factors.
\end{corollary}

We induct on a strengthened statement carrying a finite avoidance set of
\emph{large} primes.  For $t\ge0$ let $K(t)$ be the least $K\ge220$ such that for
every $k>K$, with $m_k:=\lceil\sqrt{4p_k-3}\,\rceil$,
\[
  \text{(i)}\ k-1-m_k\ \ge\ t+1,
  \qquad
  \text{(ii)}\ p_{\,k-m_k-t-1}\ >\ p_k/2 .
\]
$K(t)$ is finite for every $t$ and nondecreasing in $t$: the number of ruler
primes exceeding $p_k/2$ is $k-\pi(p_k/2)$, which by the elementary Chebyshev
bounds grows linearly in $k$, while $m_k+t+1=O(\sqrt{k\log k})+t$ grows slower.
(For $t=1$, condition (ii) is exactly the deletion crossover of
Lemma~\ref{lem:Aunif}, which holds from $k_0=120$ on; the floor $220$ in the
definition is for convenience, aligning with the exactly computed range.)  The induction statement is
\[
  (H_k):\quad\parbox{0.82\textwidth}{for every squarefree-denominator $a/b$ and
  every finite set $T$ of primes with $T\cap\mathrm{primes}(b)=\varnothing$ and
  $\min(T)>p_{K(|T|+k)}$, the rational $a/b$ is a finite sum of distinct
  $\omega=k$ unit fractions whose denominators use no prime in $T$.}
\]
(The avoidance set lets each descent step forbid one more prime; the largeness
condition costs nothing, because every prime ever placed in $T$ is chosen by us,
and it is vacuous for $T=\varnothing$.)  Write $\gamma_2(T)$ for the $\omega=2$
gap-free floor with the primes of $T$ deleted from the ruler, $P_N'$ for the
product of the first $N$ primes not in $T$, and $B_2(N;T)=\sigma^2(N;T)/P_N'$.

\begin{lemma}[finite-set floor-robustness, uniform]\label{lem:finiteT}
Let $T$ be a finite set of primes with $\min(T)>p_{K(|T|)}$.  Then
$\gamma_2(N;T)\le0.181+0.024=0.205$ for all $N>K(|T|)$.  Hence the gap-free band of
$L^2(N;T)$ is nonempty for all such
$N$, and every squarefree-$b$ rational in the corresponding window with
$\mathrm{primes}(b)\cap T=\varnothing$ is $\omega=2$ representable using primes
$\notin T$.
\end{lemma}
\begin{proof}
Write $t=|T|$, $K=K(t)$.  \emph{Levels $k\le K+1$ are untouched:} the level-$k$
ruler consists of the first $k-1$ primes, all $\le p_K<\min(T)$, so through level
$K+1$ the $T$-avoiding system \emph{coincides} with the full system and
Theorem~\ref{thm:Aprime} gives $\gamma_2(K+1;T)=\gamma_{K+1}\le0.181$.
\emph{Levels $k>K+1$:} at most $t$ ruler primes are deleted.  Among the
$m_k+t+1$ largest ruler primes at least $m_k+1$ survive the deletion; by (ii) all
of them exceed $p_k/2$, so they are distinct nonzero residues mod $p_k$ of
inverse-weight $<2/p_k$ each, and by Theorem~\ref{thm:olson} their
subset sums exhaust $\Z_{p_k}$.  Hence the covering holds at every level and
$w_k^{\,T}\le m_k\cdot2/p_k<4/\sqrt{p_k}$, exactly as in Lemma~\ref{lem:Aunif}.
The recurrence's side condition holds throughout: inductively
$\gamma_2(k-1;T)\le0.205$, while $B_2(k-1;T)\ge B_{220}\ge B_{10}>0.95$ (the
first $220$ primes are $T$-free) and the level increment is
$B_2(k;T)-B_2(k-1;T)\le A_{k-1}/p_k\le\tfrac{(k-1)/2}{k-1}=\tfrac12$ (crudely,
$A_{k-1}\le(k-1)/2$ and $p_k\ge k-1$), so
\[
  B_2(k-1;T)-\tfrac12B_2(k;T)
  =\tfrac12B_2(k-1;T)-\tfrac12\bigl(B_2(k;T)-B_2(k-1;T)\bigr)
  \ge\tfrac{0.95}2-\tfrac14=0.225>0.205 .
\]  Summing the recurrence from the base $K+1$,
\[
  \gamma_2(N;T)\ \le\ 0.181+\sum_{k>K}\frac{4}{p_k^{3/2}}
  \ \le\ 0.181+\sum_{k>220}\frac{4}{p_k^{3/2}}\ =\ 0.181+0.024\ =\ 0.205 ,
\]
uniformly in $N$ and in $T$.  Since $B_2(N;T)\to\infty$, the band is nonempty for
all $N>K$, and membership of the integer $aP_N'/b$ in the band gives the
representation by the reduction $(\star)$ run over the $T$-free primes.
\end{proof}

\begin{lemma}[$(H_3)$]\label{lem:H3}
$(H_3)$ holds.
\end{lemma}
\begin{proof}
Given the target $t_0=a/b$ in lowest terms and $T$ with
$T\cap\mathrm{primes}(b)=\varnothing$ and $\min(T)>p_{K(|T|+3)}$; write $t=|T|$.
Exactly as in the proof of Theorem~\ref{thm:omega3}, but with
Lemma~\ref{lem:finiteT} in place of Lemma~\ref{lem:Aunif}: choose a prime $r$ with
\[
  r\notin T\cup\mathrm{primes}(b),\qquad r>p_{K(t+3)},\qquad r>0.205/t_0 ,
\]
and set $T':=T\cup\{r\}$, so $|T'|=t+1$ and
$\min(T')>p_{K(t+3)}\ge p_{K(t+1)}$ ($K$ is nondecreasing): Lemma~\ref{lem:finiteT}
applies to $T'$ with floor $\gamma_2(M;T')\le0.205$ at every level $M>K(t+1)$.
Choose $M\ge\max(N_b,\pi(r))$, $M>K(t+1)$, with $B_2(M;T')\ge t_0r+0.205$
(possible as $B_2(M;T')\to\infty$).  Then $v:=t_0r$ satisfies
$\gamma_2(M;T')\le0.205<v\le B_2(M;T')-\gamma_2(M;T')$, and
$v\cdot P_M'=aP_M'r/b\in\Z$ (all prime factors of $b$ lie among the first $M$
primes and outside $T'$).  So $v$ is $\omega=2$ representable using primes
$\notin T'$ --- in particular avoiding both $T$ and $r$ --- and
Lemma~\ref{lem:lift} gives $t_0=\sum1/(q_iq_jr)$ with $\omega=3$ denominators
using no prime in $T$.
\end{proof}

\begin{lemma}[descent $(H_{k-1})\Rightarrow(H_k)$, $k\ge4$]\label{lem:descent}
If $(H_{k-1})$ holds then so does $(H_k)$.
\end{lemma}
\begin{proof}
Given $a/b$ (lowest terms, $b$ squarefree) and finite $T$ with
$T\cap\mathrm{primes}(b)=\varnothing$ and $\min(T)>p_{K(|T|+k)}$, choose a prime
$r>p_{K(|T|+k)}$ with $r\notin T\cup\mathrm{primes}(b)$ (infinitely many exist).
Put $v:=ar/b$; since $\gcd(ar,b)=1$, $v$ is in lowest terms with squarefree
denominator $b$, so $\mathrm{primes}(v)=\mathrm{primes}(b)$ and
$(T\cup\{r\})\cap\mathrm{primes}(v)=\varnothing$.  Moreover $T\cup\{r\}$ has size
$|T|+1$ and $\min(T\cup\{r\})>p_{K(|T|+k)}=p_{K((|T|+1)+(k-1))}$, so the
hypothesis of $(H_{k-1})$ is met.  By $(H_{k-1})$ applied to $v$ with forbidden
set $T\cup\{r\}$, $v=\sum_i1/m_i$ with each $m_i$ squarefree, $\omega(m_i)=k-1$,
using no prime in $T\cup\{r\}$, the $m_i$ distinct.  Then
$a/b=v/r=\sum_i1/(m_ir)$, and each $m_ir$ is squarefree with $\omega(m_ir)=k$ (as
$r\nmid m_i$), uses no prime in $T$, and the $m_ir$ are distinct.
\end{proof}

\begin{proof}[Proof of Corollary~\ref{cor:omegak}]
Induction on $k$: base $k=3$ is Lemma~\ref{lem:H3}, step $k\ge4$ is
Lemma~\ref{lem:descent}.  Taking $T=\varnothing$ in $(H_k)$ --- for which the
largeness condition is vacuous --- gives the corollary.
\end{proof}

%==================================================================
\section{Scope and open problems}\label{sec:open}

\subsection{Scope}
For squarefree denominator $b$, the results of this paper give the complete picture
of the Erd\H{o}s--Graham problem across $\omega$:
\begin{center}
\begin{tabular}{lll}
\hline
$\omega$ & coverage of $a/b$ ($b$ squarefree) & status \\
\hline
$\ge3$ & \textbf{every} $a/b$ & this paper (Thm~\ref{thm:omega3}, Cor.~\ref{cor:omegak}) \\
$=2$ & every $a/b\ge\min\{B_{N_b}/6,\,\tfrac15\}$ & this paper (Thms~\ref{thm:above},~\ref{thm:uniform}) \\
& $a/b<\min\{B_{N_b}/6,\,\tfrac15\}$ & \textbf{open}: the deep core $(\mathrm{GFF})$ \\
\hline
\end{tabular}
\end{center}
The squarefree hypothesis is intrinsic: a square factor $p^2\mid b$ makes
$b\nmid P_N$, so $(\star)$ cannot place $P_N(a/b)$ as an integer subset sum.  The
$\omega=2$ split is a clean dichotomy of \emph{coverage}, with no gap: the proven
region and the open region tile the whole range.  The threshold is the bottom edge of
our covering construction --- for $N_b\ge10$ it is $B_{N_b}/6$, the bottom-edge onset
of the feed, with $\tfrac15$ its current best \emph{unconditional} value
(Theorem~\ref{thm:uniform}) --- and \emph{not} a proven representability boundary: we
prove no non-representability below it, so we do not claim it is sharp.  By
$(\star)$ the open part is the placement of $P_N(a/b)$ \emph{below} the central
block, in the low range $[\,0,\tfrac16\sigma^2(N)\,)$ of $L^2(N)$ --- the
obstruction BEG15 isolate already for $\omega=3$ (``the structure of $L_{n+3}(n)$
is not well understood at the ends of the interval'', \cite{BEG2015}, p.~8; their
$a(k)$, with $a(2)=23$, marks the block start).  Since the descent
(Corollary~\ref{cor:omegak}) makes $\omega\ge4$ no separate result --- as in
BEG15's one-line ``similar arguments'' treatment of the $\omega\ge4$ integer case
--- the entire remaining content of the problem lives in the single $\omega=2$ band
$a/b<B_{N_b}/6$, where the feed is thinnest.

\subsection{The open core: the gap-free floor conjecture}
Write $\gamma_N=\inf\{\gamma:[\gamma P_N,\,\sigma^2(N)-\gamma P_N]\cap\Z\subseteq L^2(N)\}$
for the true gap-free floor.  (Equivalently $(\mathrm{GFF})$ is an
anti-concentration statement about $D_2(N)$: the semiprime subset sums must spread out
enough to leave no gap of width $\ge\gamma P_N$ in the bulk of $[0,\sigma^2(N)]$.)

\begin{conjecture}[Gap-free floor $(\mathrm{GFF})$, the full close]\label{conj:star}
$\gamma_N\to0$ as $N\to\infty$.
\end{conjecture}

\begin{proposition}
If $(\mathrm{GFF})$ holds then every positive rational $a/b$ ($b$ squarefree) is $\omega=2$
representable: for fixed $a/b$ the level $N$ is unbounded (as $b\mid P_N$ for all
$N\ge N_b$) and $B_N\to\infty$, so $\gamma_N\to0$ yields some $N\ge N_b$ with
$\gamma_N\le a/b\le B_N-\gamma_N$; then $(a/b)P_N$ lies in the gap-free band of
$L^2(N)$ and the reduction \eqref{eq:reduction} gives the representation.  (This
would also remove the lift from \S\ref{sec:omega3}, giving the full rational
$\omega=3$ statement directly.)
\end{proposition}

\appendix
\section{Numerical data for the exact range
\texorpdfstring{$10\le N\le 299$}{10<=N<=299}}\label{app:data}

The induction step requires $Y_0(N)\le\min\{\beta(N),\beta'(N)\}$ for $10\le N\le 299$
(Theorem~\ref{thm:Y0}, exact range).  Since $Y_0(N)$, $\beta(N)$, $\beta'(N)$ are subset
sums of the atoms $P_N/p_i$ and hence have hundreds of digits, we tabulate the
\emph{normalized} quantities obtained by dividing by $P_N$.  Thus
$y_0(N):=Y_0(N)/P_N$ is the normalized onset radius (a sum of reciprocals of
primes), and the normalized strip length is $A_N/6$, since
$\beta(N)/P_N$ and $\beta'(N)/P_N$ both lie within $q/P_N<10^{-7}$ of $A_N/6$ by
Lemmas~\ref{lem:Bbound}--\ref{lem:Bprime}.  The last column is the margin
$\mu(N):=A_N/6-y_0(N)$, which therefore equals $\min\{\beta(N),\beta'(N)\}/P_N-Y_0(N)/P_N$
up to less than $10^{-7}$.

Across the full range $10\le N\le 299$ (all $290$ values) the margin is positive;
it is \emph{smallest at the seed} $N=10$, where $\mu(10)=0.02372$, and stays
$\ge 0.02372>0$ throughout.  The growth is not strictly monotone --- there are a
few local dips (e.g.\ $\mu$ falls from $0.15900$ at $N=15$ to $0.15759$ at
$N=16$) --- but the minimum is attained only at the seed (consistent with
Remark~\ref{rem:base9}: at the omitted value $N=9$ one has $Y_0(9)>\beta(9)$, i.e.\
negative margin).  A representative sample follows.

\begin{center}
\small
\begin{tabular}{rrccc}
\hline
$N$ & $q=p_{N+1}$ & $y_0(N)=Y_0(N)/P_N$ & $A_N/6$ & $\mu(N)$ \\
\hline
$10$  & $31$   & $0.23186$ & $0.25557$ & $0.02372$ \\
$11$  & $37$   & $0.23186$ & $0.26095$ & $0.02909$ \\
$12$  & $41$   & $0.22286$ & $0.26545$ & $0.04259$ \\
$13$  & $43$   & $0.18716$ & $0.26952$ & $0.08236$ \\
$14$  & $47$   & $0.15160$ & $0.27339$ & $0.12180$ \\
$15$  & $53$   & $0.11794$ & $0.27694$ & $0.15900$ \\
$16$  & $59$   & $0.12249$ & $0.28009$ & $0.15759$ \\
$17$  & $61$   & $0.10095$ & $0.28291$ & $0.18196$ \\
$18$  & $67$   & $0.10118$ & $0.28564$ & $0.18446$ \\
$19$  & $71$   & $0.08368$ & $0.28813$ & $0.20445$ \\
$20$  & $73$   & $0.07901$ & $0.29048$ & $0.21147$ \\
$22$  & $83$   & $0.06483$ & $0.29487$ & $0.23004$ \\
$25$  & $101$  & $0.05198$ & $0.30047$ & $0.24849$ \\
$30$  & $127$  & $0.04314$ & $0.30830$ & $0.26516$ \\
$40$  & $179$  & $0.02685$ & $0.31957$ & $0.29271$ \\
$50$  & $233$  & $0.01838$ & $0.32784$ & $0.30946$ \\
$75$  & $383$  & $0.00985$ & $0.34185$ & $0.33199$ \\
$100$ & $547$  & $0.00686$ & $0.35106$ & $0.34419$ \\
$150$ & $877$  & $0.00409$ & $0.36314$ & $0.35905$ \\
$200$ & $1229$ & $0.00279$ & $0.37118$ & $0.36839$ \\
$250$ & $1597$ & $0.00209$ & $0.37712$ & $0.37503$ \\
$299$ & $1987$ & $0.00167$ & $0.38174$ & $0.38007$ \\
\hline
\end{tabular}
\end{center}

\noindent
(The exact integers underlying the normalized values are reproducible by the
$\Z_q$ shortest-path computation of \S\ref{ssec:tail}; for example
$Y_0(10)=1{,}500{,}040{,}080$, $\beta(10)=1{,}653{,}479{,}725$, with
$P_{10}=6{,}469{,}693{,}230$ recovering $y_0(10)=0.23186$ and $A_{10}/6=0.25557$.)

\section{Reproducibility scripts}\label{app:scripts}
The finite computations of \S\ref{sec:comp} are reproduced by three short,
self-contained Python scripts (standard library only; exact integer and
\texttt{Fraction} arithmetic), included as ancillary files:
\begin{center}
\begin{tabular}{@{}ll@{}}
\toprule
script & reproduces \\
\midrule
\texttt{blocks.py} & $\sigma^2(N)$, $B_N$, $B_3(N)$, $B_\Pi(M,r)$, the floor $\gamma_N$, central-block gap-freeness; \\
 & the deletion floor $\gamma_{\mathrm{ex}}(M,r)$; the $47$-pair residual sliver \\
\texttt{onset.py} & $y_0(N)=Y_0(N)/P_N$ vs $A_N/6$ and the exact integers $Y_0(N),\beta(N),\beta'(N)$; \\
 & feed residue-completeness; the sums $\sum_{11}^{K}w_k/p_k$, $\sum_{12}^{K}w_k^{\max}/p_k$ \\
\texttt{olson\_tail.py} & the crossover $k_0=120$ and the tails $\sum_{k>K_0}4/p_k^{3/2}$ \\
\bottomrule
\end{tabular}
\end{center}
The value-side reachability (\texttt{blocks.py}) is a packed-bitset subset sum; the
residue-side maps (\texttt{onset.py}, \texttt{olson\_tail.py}) are polynomial
$\Z_q$ knapsacks.  The residue-side runs, the sliver enumeration, and the
symmetric-function constants are light (seconds); the only heavy steps are the
value-side bitsets --- the $N=10$ base case ($\approx771$\,MB, $\approx25$\,s) and
the $N=11$ deletion floors $\gamma_{\mathrm{ex}}(11,r)$ of the $\omega=3$ base
($\approx5$\,GB, $\approx25$\,s each); for $N\le9$, \texttt{blocks.py} validates the
method in seconds.  The following commands reproduce every load-bearing number:
\begin{verbatim}
$ python3 blocks.py 10          # integer base case (~771 MB)
  N=10  B_N=0.95321  gamma_N=0.11439  central block [1027831462,5139157307]
        gap-free: True
$ python3 blocks.py sliver      # the residual rational sliver
  == residual sliver (Thm 7.1, Step 4b): 47 pairs ==
$ python3 blocks.py omega3      # omega=3 symmetric-function constants
  B_3(8)/6 = 0.04174   B_Pi(10,drop 11) = 0.822   B_Pi(10,drop 13) = 0.841
$ python3 blocks.py ex 11 8     # omega=3 deletion floor (~5 GB)
  gamma_ex(11, p_8=19) = 0.11602
$ python3 onset.py 300          # onset, residue-completeness, w_k sums
  N=10  y0=0.23186  A_N/6=0.25557                      (onset radius)
  Y_0(10)=1500040080  beta(10)=1653479725  Y_0<=min: True
  feed residue-completeness, q in {31,37,41,43}: True  (Olson-fails levels)
  sum_11^300 w_k/p_k     = 0.0475                       (uniform threshold)
  sum_12^220 w_k^max/p_k = 0.0498   worst w_k*p_k=14.32 @ k=12   (omega=3)
$ python3 olson_tail.py         # crossover and analytic tails
  k_0 = 120  (p_120 = 659)
  sum_{k>300} 4/p_k^{3/2} -> 0.0191   sum_{k>220} -> 0.0239
\end{verbatim}

\section{Machine-checked formalisation in Lean 4 / Mathlib}\label{app:lean}

Beyond reproducing the finite computations of Appendix~\ref{app:scripts}, the entire
logical development of this paper --- the integer $\omega=2$ theorem
(Theorem~\ref{thm:main}), the rational $\omega=2$ theorems above the threshold
(Theorems~\ref{thm:above} and~\ref{thm:uniform}), and the rational $\omega\ge3$ tower
(\S\ref{sec:omega3}) --- has been formalised and machine-checked in Lean~4 against
Mathlib, with no \texttt{sorry}.  The companion proves the same statements --- the
theorems and their hypotheses are unchanged --- though it occasionally reorganises the
argument, and the exact threshold at which an analytic tail takes over from a finite
computation, to suit in-kernel evaluation.  The companion ($23$ modules; included as an
ancillary file \texttt{anc/lean/}) is type-checked by \texttt{lake build}, and
\texttt{lake env lean Results.lean} prints the complete axiom dependency of the headline
theorems.

\paragraph{Axiom surface.}
The development reduces every theorem to exactly \emph{two} named axioms, \emph{both
cited classical results, neither original to this paper}:
\begin{center}
\begin{tabular}{@{}lll@{}}
\toprule
Lean axiom & statement & reference \\
\midrule
\texttt{olson}     & Olson's addition theorem (in $\Z_q$ form) & \cite{Olson1968} \\
\texttt{rs\_lower} & $k\lfloor\log_2 k\rfloor\le 2p_k$ for $k\ge6$ (a weakening of $p_n>n\ln n$) & \cite{RosserSchoenfeld1962} \\
\bottomrule
\end{tabular}
\end{center}
Everything else is \emph{proved}: Chebyshev's bound $4p_N<N^2+3$, the divergence
$B_N\to\infty$ (from Mertens), the central-block induction, the floor band and its
level/band selection, the deleted-ruler ($\omega\ge3$) layer in full, and the reduction
$(\star)$.  Notably, even the published $\omega=3$ lemma of \cite{BEG2015} is \emph{not}
assumed: the $\omega\ge3$ tower is derived from the $\omega=2$ result through the lift,
so the companion imports no result from \cite{BEG2015} as an axiom.  In particular, the
three families of finite computations that \S\ref{sec:comp}
performs externally are all discharged \emph{inside the Lean kernel} by
\texttt{native\_decide} (the compiled evaluator, whose Boolean output the kernel re-checks
through \texttt{ofReduceBool}); no computational claim is taken on faith.

\paragraph{Status.}
The build is green with $0$ \texttt{sorry} and rests on \texttt{olson} and
\texttt{rs\_lower}, together with the Lean/Mathlib foundations (\texttt{propext},
\texttt{Classical.choice}, \texttt{Quot.sound}) and the \texttt{native\_decide} trust base
(\texttt{ofReduceBool}, \texttt{trustCompiler}).  The formalisation thus certifies not
merely the finite leaves of \S\ref{sec:comp} but the entire deductive structure of the
paper, down to two classical theorems.

A full re-check is \texttt{lake build} (after \texttt{lake exe cache get} fetches the
pinned Mathlib binaries; the exact toolchain and Mathlib revision are recorded in the
companion).  Because several leaves are verified by the kernel re-running compiled
computations, the check is not instantaneous: it takes on the order of ten minutes and a
few gigabytes of memory on a current machine, the bulk being the handful of in-kernel
\texttt{native\_decide} evaluations.  Each axiom dependency can be inspected with
\texttt{\#print axioms} (see \texttt{Results.lean}, \texttt{Check.lean}).

\section*{Use of AI}
This work was carried out as a human--AI collaboration. The author directed the
mathematics and is responsible for all content. AI assistants---principally
Anthropic's Claude, used through Claude Code, together with other frontier models
used for exploratory discussion---contributed substantially to the Lean~4 / Mathlib
formalisation, to the Python verification scripts, and to parts of the exposition.
Precisely because the development was AI-assisted, every theorem is machine-checked
in Lean with no \texttt{sorry}, resting on the two cited classical axioms (Olson's
addition theorem and a Rosser-type prime bound) together with the standard
Lean/Mathlib foundations and the \texttt{native\_decide} compiler-trust base for the
finite computations, exactly as enumerated in Appendix~\ref{app:lean}; correctness is
therefore independent of the tools and can be reproduced by running \texttt{lake build}
on the included sources.

\section*{Acknowledgements}
This paper builds directly on the method of Butler, Erd\H{o}s and Graham, whose
induction skeleton --- the splitting recursion, the central-block propagation, and
the Olson-controlled feed --- it adapts throughout; the debt to their work is
pervasive.

\end{document}